\documentclass[10pt]{article}
\usepackage{amsmath}
\usepackage{amsfonts}
\usepackage{amssymb}
\usepackage{amsthm}
\usepackage[margin=2.5cm, vmargin={1.5cm},includefoot]{geometry}

\def\ga{{G_0}}  
\def\gb{{G_1}}  
\def\gc{{G_2}}  

\usepackage{pst-node}
\usepackage{pst-coil}
\usepackage{pst-plot}
\usepackage{pst-grad}
\usepackage{pstricks}

\newrgbcolor{LemonChiffon}{1.0 0.98 0.8}
\newrgbcolor{magenta}{1.0 0.3 0.4}
\newrgbcolor{mistyrose}{1.0 0.0 0.6}
\newrgbcolor{deeppink}{0.7 0.98 0.08}
\newrgbcolor{lightsalmon}{1.0 0.63 0.48}
\newrgbcolor{lightred}{0.7 0.0 0.0}
\newrgbcolor{lightblue}{0.68 0.85 0.90}
\newrgbcolor{indianred}{0.80  0.36  0.36}
\newrgbcolor{lightgreen}{0.56 0.93 0.56}
\newrgbcolor{byellow}{0.93 0.76 0.3}
\newrgbcolor{brown}{0.6 0.3 0.5}

\title{A  Diophantine Frobenius problem \\ related to Riemann surfaces}
\author{Cormac O'Sullivan and Anthony Weaver}
\date{December 12, 2009}
\begin{document}

\maketitle

\newtheorem{theorem}{Theorem}[section]
\newtheorem{lemma}[theorem]{Lemma}
\newtheorem{prop}[theorem]{Proposition}
\newtheorem{cor}[theorem]{Corollary}
\newtheorem{conj}{Conjecture}
\newtheorem{rem}{Remark}

\renewcommand{\labelenumi}{(\roman{enumi})}
\def\fr{{\frak f}}

\numberwithin{equation}{section}

\bibliographystyle{plain}
\begin{abstract}   We obtain  sharp upper and lower bounds on a certain four-dimensional Frobenius number determined by a prime pair $(p,q)$, $2<p<q$,  including  exact formulae for two infinite subclasses of such pairs.  Our work is motivated by the study of compact Riemann surfaces which can be realized as a semi-regular $pq$-fold coverings of surfaces of lower genus.  In this context,  the Frobenius number is (up to an additive translation) the largest genus in which no surface is such a covering.   In many cases it is also the largest genus in which no surface admits an automorphism of order $pq$.    The general $t$-dimensional Frobenius problem ($t \geq 3$) is $NP$-hard, and it  may be that our restricted problem retains this property.
\end{abstract}

\section{Introduction}\label{S:intro}

A set of  integers $\{a_1, a_2, \dots a_t\}$,  $t \geq 2$, with $a_i > 1$ and gcd $=1$, has a {\it Frobenius number}
$$g(\{a_1, a_2, \dots, a_t\}),$$
  which is the largest positive integer not representable in the form $k_1a_1 + k_2a_2 + \dots + k_ta_t$,  where each $k_i$ is a nonnegative integer.  It is a simple exercise to show that $g(\{a_1, a_2, \dots, a_t\})$ exists under the stated conditions.  Finding $g(\{a_1, \dots, a_t\})$ for a given set $\{a_1, \dots, a_t\}$ is the  linear Diophantine problem of Frobenius \cite{JLRA05}.    In 1884, J.J. Sylvester established the formula
\begin{equation}\label{E:sylvester}
g(\{a_1,a_2\})= a_1a_2 - a_1 -a_2
\end{equation}
 for the {\it two-dimensional} Frobenius number  \cite{JJS84}.   In 1990, it was shown by F. Curtis \cite{FC90} that, for $t\geq 3$,  there is no finite set of polynomials $\{f_1, \dots, f_k\}$ in $t$ variables such that, for each $t$-tuple $\{a_1, a_2,  \dots, a_t\}$ with greatest common divisor $1$,  $g(\{a_1, a_2,  \dots, a_t\}) = f_i(a_1, a_2, \dots, a_t)$ for some $i$.    Algorithms for computing $t$-dimensional Frobenius numbers  exist  \cite{JLRA05}, but the problem (for variable $t\geq 3$)  is $NP$-hard \cite{JLRA96}.

Throughout the paper, $p$, $q$ will be primes satisfying $2 < p <q$ with   $p'$, $q'$ denoting the integers $(p-1)/2$ and $(q-1)/2$, respectively.  The four integers
\begin{equation}\label{E:d's}
d_0=pq,\quad d_1=p'q, \quad d_2=pq', \quad d_3=(pq-1)/2,
\end{equation}
have gcd $=1$, so they determine a four-dimensional Frobenius number
\begin{equation}\label{E:defg}
g_{pq}= g(\{d_0, d_1, d_2, d_3\}).
\end{equation}
The significance of the number $g_{pq}$ in \eqref{E:defg} is that
$$
g_{pq} -pq+1
$$
is the largest  integer  such that no compact Riemann surface of that genus  is a semi-regular $pq$-fold cover  of some other surface. This is explained in Section~\ref{S:motivation}. A closely related quantity of interest to us is $\nu_{pq}$, the largest  integer  such that no compact Riemann surface of that genus has an automorphism group that is cyclic of order $pq$.   $\nu_{pq}$ is called the {\it largest non-genus}  of the group $\Bbb Z_{pq}$. As a special case of Theorem \ref{gnnn} we have
\begin{equation}\label{gpqbound}
g_{pq} -pq +1 \leq \nu_{pq} \leq g_{pq}.
\end{equation}
Our main results, listed in the next section, yield bounds for $g_{pq}$. When $q$ is sufficiently large with respect to $p$ we obtain exact formulas for $g_{pq}$ as well as $\nu_{pq}$.   At the other extreme we also give exact formulas for $g_{pq}$, $\nu_{pq}$ when $q=p+2$.

More generally, as we describe in Section~\ref{S:motivation}, there is a Frobenius number $g_n$ so that $g_n-n+1$ is the largest possible genus for a compact Riemann surface that is not a semi-regular $n$-fold cover  of another surface. For square-free odd $n$ with $s>2$ prime factors, this will correspond to a more difficult $2^s$-dimensional Frobenius problem. $\nu_{n}$, the largest non-genus for the cyclic group $\Bbb Z_{n}$,  has been found in the case of $n=p^e$ for $p$ prime by Kulkarni and Maclachlan in \cite{KM91}. Kulkarni in \cite{K87} showed that, for an arbitrary finite group $G$, the genera where it is possible for a surface to admit $G$ as an automorphism group form an arithmetic progression. He showed that there also exists a largest non-genus in this progression. These genera are studied with generating functions in \cite{MM98}.

\section{The  main results}\label{S:results}

Define the function
\begin{equation}\label{E:definefracpq}
\fr_{p,q}(x,y,z,w) = xd_0 + yd_1 + zd_2 + wd_3,
\end{equation}
where the integers $d_i$ are defined at \eqref{E:d's}.   A  positive integer $n$ is {\it representable} if $n = \fr_{p,q}(x,y,z,w)$ for $x,y,z,w$ nonnegative.   The Frobenius number \eqref{E:defg} is the largest  non-representable integer.    Since $p$ and $q$ are fixed in all our arguments,  we henceforth put $\fr = \fr_{p,q}$, suppressing the subscripts and write $g$ for $g_{pq}$.

We define integers $\kappa, \kappa', \lambda, \lambda'$ as follows:
\begin{alignat}{5}
q&=\kappa p + \lambda, &\qquad & 1 &\leq \lambda &\leq p-1\label{E:defs1} \\
q' &= \kappa' p' + \lambda ',&\qquad &  0 & \leq \lambda'& \leq p'-1. \label{E:defs2}
\end{alignat}
The  integers
\begin{equation}\label{E:keynums}
\ga \equiv \fr(p'-1, p-1, \kappa, -1), \qquad
 \gb \equiv   \ga-\lambda d_3, \qquad
 \gc \equiv   \ga-(p-3) d_3
\end{equation}
play an important role.

\begin{theorem}\label{T1:Main}  The Frobenius number $g$ satisfies
\begin{enumerate}
\item $\gc \leq g \leq \ga$
\item $g = \ga$ if and only if $\kappa + \lambda \geq p$
\item $g=\gc$ if $p=3$ or if $(p,q)$ is a twin prime pair.
\end{enumerate}
\end{theorem}

We note that if $p=3$ then $\gc=\ga$ and $\kappa + \lambda \geq 3$, so that Theorem~\ref{T1:Main} parts (i) and (ii) each imply that $g=\ga$ (and hence also $g=\gc$, as in part (iii)).
 For $p>3$,  the integer $q = (p-3)p + 1$, if prime, is the largest  such that $\kappa + \lambda < p$.  Hence we obtain an easy corollary.

 \begin{cor}  If $q \geq  (p-3)p + 3$, $g=\ga$.
 \end{cor}

When $\kappa + \lambda < p$, by Theorem~\ref{T1:Main}, $\gc \leq g<\ga$.  These bounds can be tightened in some cases.  To treat these cases, we introduce some more notation.

Note that $\kappa$ and $\lambda$ have opposite parity (otherwise $q$ is not prime), and that $\kappa'\geq \kappa$.     If $\kappa + \lambda < p$, then, in fact, $\kappa + \lambda \leq p-2$ and hence $\lambda \leq p-3$.  It follows that there is a unique nonnegative integer $\tau < \lambda$  such that
\begin{equation}\label{E:deftau}
\frac{\tau+2}{\tau+1} < \frac{p}{\lambda} < \frac{\tau+1}{\tau}.
\end{equation}
We allow $\tau=0$ so as to include the cases in which $2 < \frac{p}{\lambda}$.  (It is also easy to see that $\tau=\lfloor \lambda/(p-\lambda)\rfloor$.) Every pair $(p,q)$ with $\kappa + \lambda < p$ belongs to one of two types:
\begin{equation}\label{E:typesdef}
\text{Type I:}\quad \frac{\tau+2}{\tau+1} <\frac{p'}{\lambda'}\qquad\quad \text{Type II:}\quad\frac{p'}{\lambda'} \leq  \frac{\tau+2}{\tau+1}.
\end{equation}

\begin{theorem}\label{T4:Main}  For a Type II pair, the Frobenius number $g$ satisfies
\begin{enumerate}
\item $\gb \leq g < \ga$
\item $g=\gb$
if and only if  $\kappa + \lambda \leq p-\lambda$.
\end{enumerate}
\end{theorem}

\begin{theorem}\label{T5:Main}  For a Type I pair with $\kappa + \lambda \leq p-\lambda$, the Frobenius number $g$ satisfies
\begin{enumerate}
\item    $\gc \leq g < \gb$
\item $g=\gc$ if $(p,q)$ is a twin prime pair.
\end{enumerate}
\end{theorem}

The above theorems show where $g$ lies in relation to $\ga$, $\gb$ and $\gc$.   Figure 1 shows how these results are distributed over small prime pairs. The four displayed cases correspond to $g=\ga$, $g=\gb$, $\gb<g<\ga$ and $\gc \leq g <\gb$, respectively.




\SpecialCoor


\psset{griddots=5,subgriddiv=0,gridlabels=0pt}

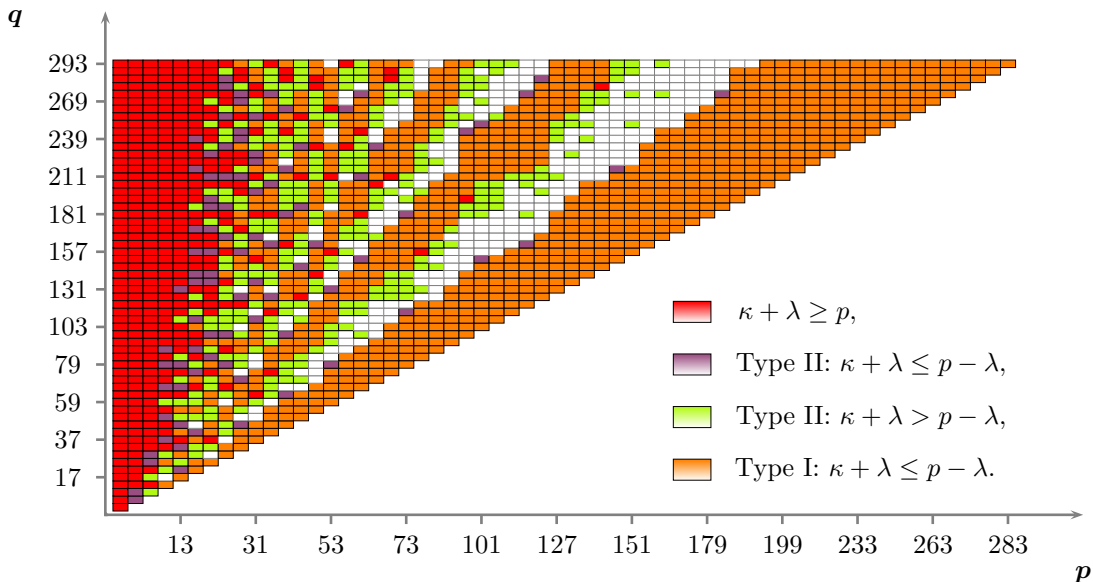
\begin{figure}[h]
\begin{center}
\begin{pspicture}(0.2,0)(14.8,7.5) 

\psset{linewidth=1pt}
\psaxes[labels=none,Dx=1,Dy=0.5,linecolor=gray, tickstyle=bottom]{->}(1,0.5)(1,0.5)(14,7.2)

\psset{linewidth=0.5pt}

\def\ddota{
\pscustom[fillstyle=gradient,gradmidpoint=1,gradbegin=red,gradend=white,gradlines=170]{%
  \psframe(0.25,0.14)(-0.25,-0.14)}}

\def\ddotd{
\pscustom[fillstyle=gradient,gradmidpoint=1,gradbegin=orange,gradend=white,gradlines=170]{%
  \psframe(0.25,0.14)(-0.25,-0.14)}}

\def\ddotc{
\pscustom[fillstyle=gradient,gradmidpoint=1,gradbegin=deeppink,gradend=white,gradlines=170]{%
  \psframe(0.25,0.14)(-0.25,-0.14)}}

\def\ddotb{
\pscustom[fillstyle=gradient,gradmidpoint=1,gradbegin=brown,gradend=white,gradlines=170]{%
  \psframe(0.25,0.14)(-0.25,-0.14)}}

\def\ddote{
\pscustom[fillstyle=gradient,gradmidpoint=1,linecolor=gray, gradbegin=white,gradend=white,gradlines=170]{%
  \psframe(0.25,0.14)(-0.25,-0.14)}}

\psset{linewidth=0pt}

\def\dota{
\pscustom[fillstyle=solid,fillcolor=red]{%
  \psframe(0.1,0.05)(-0.1,-0.05)}}

\def\dotd{
\pscustom[fillstyle=solid,fillcolor=orange]{%
  \psframe(0.1,0.05)(-0.1,-0.05)}}

\def\dotc{
\pscustom[fillstyle=solid,fillcolor=deeppink]{%
  \psframe(0.1,0.05)(-0.1,-0.05)}}

\def\dotb{
\pscustom[fillstyle=solid,fillcolor=brown]{%
  \psframe(0.1,0.05)(-0.1,-0.05)}}

\def\dote{
\pscustom[fillstyle=gradient,gradmidpoint=1,linecolor=gray, gradbegin=white,gradend=white,gradlines=170]{%
  \psframe(0.1,0.05)(-0.1,-0.05)}}

\rput(-0.2,7.1){{\boldmath $q$}}
\rput(14,-0.3){{\boldmath $p$}}

  \rput(2,0.1){$13$}
  \rput(3,0.1){$31$}
  \rput(4,0.1){$53$}
  \rput(5,0.1){$73$}
  \rput(6,0.1){$101$}
  \rput(7,0.1){$127$}
  \rput(8,0.1){$151$}
  \rput(9,0.1){$179$}
  \rput(10,0.1){$199$}
  \rput(11,0.1){$233$}
  \rput(12,0.1){$263$}
  \rput(13,0.1){$283$}


  \rput(0.5,1){$17$}
  \rput(0.5,1.5){$37$}
  \rput(0.5,2){$59$}
  \rput(0.5,2.5){$79$}
  \rput(0.5,3){$103$}
  \rput(0.5,3.5){$131$}
  \rput(0.5,4){$157$}
  \rput(0.5,4.5){$181$}
  \rput(0.5,5){$211$}
  \rput(0.5,5.5){$239$}
  \rput(0.5,6){$269$}
  \rput(0.5,6.5){$293$}

\rput(10.2,3.2){$\kappa+\lambda \geq p$,}
\rput(11.18,2.5){Type II: $\kappa+\lambda \leq p-\lambda$,}
\rput(11.18,1.8){Type II: $\kappa+\lambda > p-\lambda$,}
\rput(11.1,1.1){Type I: $\kappa+\lambda \leq p-\lambda$.}

\multips(8.8,3.2)(0,0.5){1}{\ddota}
\rput(8.8,2.5){\ddotb}
\multips(8.8,1.8)(0,0.5){1}{\ddotc}
\multips(8.8,1.1)(0,0.5){1}{\ddotd}

\multips(1.2,0.6)(0.2,0){1}{\dota}

\multips(1.2,0.7)(0.2,0){1}{\dota}

\multips(1.4,0.7)(0.2,0){1}{\dotb}

\multips(1.2,0.8)(0.2,0){1}{\dota}

\multips(1.4,0.8)(0.2,0){1}{\dotb}

\multips(1.6,0.8)(0.2,0){1}{\dotc}

\multips(1.2,0.9)(0.2,0){3}{\dota}

\multips(1.8,0.9)(0.2,0){1}{\dotd}

\multips(1.2,1.)(0.2,0){2}{\dota}

\multips(1.6,1.)(0.2,0){1}{\dotc}

\multips(1.8,1.)(0.2,0){1}{\dote}

\multips(2.,1.)(0.2,0){1}{\dotd}

\multips(1.2,1.1)(0.2,0){3}{\dota}

\multips(1.8,1.1)(0.2,0){1}{\dotc}

\multips(2.,1.1)(0.2,0){1}{\dotb}

\multips(2.2,1.1)(0.2,0){1}{\dotd}

\multips(1.2,1.2)(0.2,0){2}{\dota}

\multips(1.6,1.2)(0.2,0){1}{\dotb}

\multips(1.8,1.2)(0.2,0){1}{\dotd}

\multips(2.,1.2)(0.2,0){1}{\dotc}

\multips(2.2,1.2)(0.2,0){2}{\dotd}

\multips(1.2,1.3)(0.2,0){2}{\dota}

\multips(1.6,1.3)(0.2,0){1}{\dotb}

\multips(1.8,1.3)(0.2,0){1}{\dotc}

\multips(2.,1.3)(0.2,0){1}{\dotd}

\multips(2.2,1.3)(0.2,0){1}{\dotc}

\multips(2.4,1.3)(0.2,0){1}{\dote}

\multips(2.6,1.3)(0.2,0){1}{\dotd}

\multips(1.2,1.4)(0.2,0){4}{\dota}

\multips(2.,1.4)(0.2,0){1}{\dotb}

\multips(2.2,1.4)(0.2,0){2}{\dotc}

\multips(2.6,1.4)(0.2,0){2}{\dotd}

\multips(1.2,1.5)(0.2,0){3}{\dota}

\multips(1.8,1.5)(0.2,0){1}{\dotb}

\multips(2.,1.5)(0.2,0){1}{\dota}

\multips(2.2,1.5)(0.2,0){1}{\dotd}

\multips(2.4,1.5)(0.2,0){1}{\dota}

\multips(2.6,1.5)(0.2,0){1}{\dote}

\multips(2.8,1.5)(0.2,0){2}{\dotd}

\multips(1.2,1.6)(0.2,0){4}{\dota}

\multips(2.,1.6)(0.2,0){1}{\dotd}

\multips(2.2,1.6)(0.2,0){1}{\dotb}

\multips(2.4,1.6)(0.2,0){1}{\dotd}

\multips(2.6,1.6)(0.2,0){1}{\dotc}

\multips(2.8,1.6)(0.2,0){3}{\dotd}

\multips(1.2,1.7)(0.2,0){4}{\dota}

\multips(2.,1.7)(0.2,0){1}{\dotb}

\multips(2.2,1.7)(0.2,0){1}{\dote}

\multips(2.4,1.7)(0.2,0){1}{\dotd}

\multips(2.6,1.7)(0.2,0){1}{\dotc}

\multips(2.8,1.7)(0.2,0){1}{\dotb}

\multips(3.,1.7)(0.2,0){3}{\dotd}

\multips(1.2,1.8)(0.2,0){3}{\dota}

\multips(1.8,1.8)(0.2,0){1}{\dotb}

\multips(2.,1.8)(0.2,0){3}{\dotc}

\multips(2.6,1.8)(0.2,0){1}{\dotd}

\multips(2.8,1.8)(0.2,0){2}{\dote}

\multips(3.2,1.8)(0.2,0){3}{\dotd}

\multips(1.2,1.9)(0.2,0){4}{\dota}

\multips(2.,1.9)(0.2,0){2}{\dotd}

\multips(2.4,1.9)(0.2,0){1}{\dotc}

\multips(2.6,1.9)(0.2,0){1}{\dotd}

\multips(2.8,1.9)(0.2,0){1}{\dotc}

\multips(3.,1.9)(0.2,0){1}{\dote}

\multips(3.2,1.9)(0.2,0){4}{\dotd}

\multips(1.2,2.)(0.2,0){3}{\dota}

\multips(1.8,2.)(0.2,0){3}{\dotc}

\multips(2.4,2.)(0.2,0){1}{\dotd}

\multips(2.6,2.)(0.2,0){1}{\dote}

\multips(2.8,2.)(0.2,0){1}{\dotd}

\multips(3.,2.)(0.2,0){1}{\dotc}

\multips(3.2,2.)(0.2,0){1}{\dote}

\multips(3.4,2.)(0.2,0){4}{\dotd}

\multips(1.2,2.1)(0.2,0){5}{\dota}

\multips(2.2,2.1)(0.2,0){1}{\dotc}

\multips(2.4,2.1)(0.2,0){1}{\dotd}

\multips(2.6,2.1)(0.2,0){1}{\dotc}

\multips(2.8,2.1)(0.2,0){1}{\dotd}

\multips(3.,2.1)(0.2,0){1}{\dota}

\multips(3.2,2.1)(0.2,0){1}{\dotc}

\multips(3.4,2.1)(0.2,0){1}{\dotb}

\multips(3.6,2.1)(0.2,0){4}{\dotd}

\multips(1.2,2.2)(0.2,0){3}{\dota}

\multips(1.8,2.2)(0.2,0){2}{\dotb}

\multips(2.2,2.2)(0.2,0){1}{\dota}

\multips(2.4,2.2)(0.2,0){1}{\dotc}

\multips(2.6,2.2)(0.2,0){1}{\dota}

\multips(2.8,2.2)(0.2,0){2}{\dotd}

\multips(3.2,2.2)(0.2,0){1}{\dotc}

\multips(3.4,2.2)(0.2,0){2}{\dote}

\multips(3.8,2.2)(0.2,0){4}{\dotd}

\multips(1.2,2.3)(0.2,0){4}{\dota}

\multips(2.,2.3)(0.2,0){1}{\dotc}

\multips(2.2,2.3)(0.2,0){1}{\dotd}

\multips(2.4,2.3)(0.2,0){1}{\dotc}

\multips(2.6,2.3)(0.2,0){1}{\dotd}

\multips(2.8,2.3)(0.2,0){1}{\dotb}

\multips(3.,2.3)(0.2,0){1}{\dotd}

\multips(3.2,2.3)(0.2,0){3}{\dotc}

\multips(3.8,2.3)(0.2,0){1}{\dote}

\multips(4.,2.3)(0.2,0){4}{\dotd}

\multips(1.2,2.4)(0.2,0){5}{\dota}

\multips(2.2,2.4)(0.2,0){1}{\dotb}

\multips(2.4,2.4)(0.2,0){1}{\dota}

\multips(2.6,2.4)(0.2,0){1}{\dotd}

\multips(2.8,2.4)(0.2,0){1}{\dote}

\multips(3.,2.4)(0.2,0){1}{\dotd}

\multips(3.2,2.4)(0.2,0){1}{\dota}

\multips(3.4,2.4)(0.2,0){1}{\dotc}

\multips(3.6,2.4)(0.2,0){2}{\dote}

\multips(4.,2.4)(0.2,0){5}{\dotd}

\multips(1.2,2.5)(0.2,0){3}{\dota}

\multips(1.8,2.5)(0.2,0){2}{\dotb}

\multips(2.2,2.5)(0.2,0){1}{\dotc}

\multips(2.4,2.5)(0.2,0){1}{\dotd}

\multips(2.6,2.5)(0.2,0){1}{\dotb}

\multips(2.8,2.5)(0.2,0){1}{\dotc}

\multips(3.,2.5)(0.2,0){1}{\dote}

\multips(3.2,2.5)(0.2,0){1}{\dotd}

\multips(3.4,2.5)(0.2,0){2}{\dotc}

\multips(3.8,2.5)(0.2,0){1}{\dote}

\multips(4.,2.5)(0.2,0){1}{\dotb}

\multips(4.2,2.5)(0.2,0){5}{\dotd}

\multips(1.2,2.6)(0.2,0){4}{\dota}

\multips(2.,2.6)(0.2,0){1}{\dotc}

\multips(2.2,2.6)(0.2,0){1}{\dota}

\multips(2.4,2.6)(0.2,0){1}{\dotb}

\multips(2.6,2.6)(0.2,0){2}{\dotc}

\multips(3.,2.6)(0.2,0){1}{\dote}

\multips(3.2,2.6)(0.2,0){2}{\dotd}

\multips(3.6,2.6)(0.2,0){1}{\dotc}

\multips(3.8,2.6)(0.2,0){2}{\dote}

\multips(4.2,2.6)(0.2,0){6}{\dotd}

\multips(1.2,2.7)(0.2,0){3}{\dota}

\multips(1.8,2.7)(0.2,0){1}{\dotb}

\multips(2.,2.7)(0.2,0){1}{\dota}

\multips(2.2,2.7)(0.2,0){1}{\dotb}

\multips(2.4,2.7)(0.2,0){1}{\dotc}

\multips(2.6,2.7)(0.2,0){1}{\dota}

\multips(2.8,2.7)(0.2,0){1}{\dotd}

\multips(3.,2.7)(0.2,0){1}{\dotc}

\multips(3.2,2.7)(0.2,0){3}{\dotd}

\multips(3.8,2.7)(0.2,0){1}{\dotc}

\multips(4.,2.7)(0.2,0){2}{\dote}

\multips(4.4,2.7)(0.2,0){6}{\dotd}

\multips(1.2,2.8)(0.2,0){6}{\dota}

\multips(2.4,2.8)(0.2,0){4}{\dotd}

\multips(3.2,2.8)(0.2,0){1}{\dotc}

\multips(3.4,2.8)(0.2,0){3}{\dotd}

\multips(4.,2.8)(0.2,0){1}{\dotc}

\multips(4.2,2.8)(0.2,0){2}{\dote}

\multips(4.6,2.8)(0.2,0){6}{\dotd}

\multips(1.2,2.9)(0.2,0){6}{\dota}

\multips(2.4,2.9)(0.2,0){2}{\dotb}

\multips(2.8,2.9)(0.2,0){1}{\dotc}

\multips(3.,2.9)(0.2,0){1}{\dotd}

\multips(3.2,2.9)(0.2,0){1}{\dotc}

\multips(3.4,2.9)(0.2,0){1}{\dotb}

\multips(3.6,2.9)(0.2,0){2}{\dotd}

\multips(4.,2.9)(0.2,0){1}{\dotc}

\multips(4.2,2.9)(0.2,0){1}{\dote}

\multips(4.4,2.9)(0.2,0){1}{\dotc}

\multips(4.6,2.9)(0.2,0){1}{\dote}

\multips(4.8,2.9)(0.2,0){6}{\dotd}

\multips(1.2,3.)(0.2,0){5}{\dota}

\multips(2.2,3.)(0.2,0){1}{\dotd}

\multips(2.4,3.)(0.2,0){2}{\dotc}

\multips(2.8,3.)(0.2,0){1}{\dote}

\multips(3.,3.)(0.2,0){1}{\dotd}

\multips(3.2,3.)(0.2,0){1}{\dotc}

\multips(3.4,3.)(0.2,0){1}{\dote}

\multips(3.6,3.)(0.2,0){2}{\dotd}

\multips(4.,3.)(0.2,0){2}{\dotc}

\multips(4.4,3.)(0.2,0){2}{\dote}

\multips(4.8,3.)(0.2,0){7}{\dotd}

\multips(1.2,3.1)(0.2,0){4}{\dota}

\multips(2.,3.1)(0.2,0){1}{\dotc}

\multips(2.2,3.1)(0.2,0){1}{\dotb}

\multips(2.4,3.1)(0.2,0){3}{\dotc}

\multips(3.,3.1)(0.2,0){1}{\dotb}

\multips(3.2,3.1)(0.2,0){1}{\dotc}

\multips(3.4,3.1)(0.2,0){1}{\dote}

\multips(3.6,3.1)(0.2,0){1}{\dotc}

\multips(3.8,3.1)(0.2,0){2}{\dotd}

\multips(4.2,3.1)(0.2,0){4}{\dote}

\multips(5.,3.1)(0.2,0){7}{\dotd}

\multips(1.2,3.2)(0.2,0){5}{\dota}

\multips(2.2,3.2)(0.2,0){1}{\dotc}

\multips(2.4,3.2)(0.2,0){1}{\dota}

\multips(2.6,3.2)(0.2,0){2}{\dotc}

\multips(3.,3.2)(0.2,0){1}{\dote}

\multips(3.2,3.2)(0.2,0){1}{\dota}

\multips(3.4,3.2)(0.2,0){1}{\dotc}

\multips(3.6,3.2)(0.2,0){1}{\dote}

\multips(3.8,3.2)(0.2,0){2}{\dotd}

\multips(4.2,3.2)(0.2,0){2}{\dotc}

\multips(4.6,3.2)(0.2,0){2}{\dote}

\multips(5.,3.2)(0.2,0){1}{\dotb}

\multips(5.2,3.2)(0.2,0){7}{\dotd}

\multips(1.2,3.3)(0.2,0){9}{\dota}

\multips(3.,3.3)(0.2,0){1}{\dotc}

\multips(3.2,3.3)(0.2,0){1}{\dotd}

\multips(3.4,3.3)(0.2,0){2}{\dotc}

\multips(3.8,3.3)(0.2,0){2}{\dotd}

\multips(4.2,3.3)(0.2,0){2}{\dotc}

\multips(4.6,3.3)(0.2,0){3}{\dote}

\multips(5.2,3.3)(0.2,0){8}{\dotd}

\multips(1.2,3.4)(0.2,0){5}{\dota}

\multips(2.2,3.4)(0.2,0){1}{\dotc}

\multips(2.4,3.4)(0.2,0){1}{\dota}

\multips(2.6,3.4)(0.2,0){1}{\dotc}

\multips(2.8,3.4)(0.2,0){1}{\dotb}

\multips(3.,3.4)(0.2,0){1}{\dotd}

\multips(3.2,3.4)(0.2,0){1}{\dotb}

\multips(3.4,3.4)(0.2,0){1}{\dotd}

\multips(3.6,3.4)(0.2,0){1}{\dota}

\multips(3.8,3.4)(0.2,0){1}{\dote}

\multips(4.,3.4)(0.2,0){3}{\dotd}

\multips(4.6,3.4)(0.2,0){3}{\dotc}

\multips(5.2,3.4)(0.2,0){2}{\dote}

\multips(5.6,3.4)(0.2,0){7}{\dotd}

\multips(1.2,3.5)(0.2,0){4}{\dota}

\multips(2.,3.5)(0.2,0){1}{\dotb}

\multips(2.2,3.5)(0.2,0){2}{\dota}

\multips(2.6,3.5)(0.2,0){1}{\dotc}

\multips(2.8,3.5)(0.2,0){1}{\dote}

\multips(3.,3.5)(0.2,0){1}{\dotd}

\multips(3.2,3.5)(0.2,0){1}{\dote}

\multips(3.4,3.5)(0.2,0){2}{\dotd}

\multips(3.8,3.5)(0.2,0){1}{\dotc}

\multips(4.,3.5)(0.2,0){1}{\dotb}

\multips(4.2,3.5)(0.2,0){2}{\dotd}

\multips(4.6,3.5)(0.2,0){4}{\dotc}

\multips(5.4,3.5)(0.2,0){1}{\dote}

\multips(5.6,3.5)(0.2,0){8}{\dotd}

\multips(1.2,3.6)(0.2,0){5}{\dota}

\multips(2.2,3.6)(0.2,0){2}{\dotb}

\multips(2.6,3.6)(0.2,0){1}{\dota}

\multips(2.8,3.6)(0.2,0){1}{\dotc}

\multips(3.,3.6)(0.2,0){1}{\dotb}

\multips(3.2,3.6)(0.2,0){1}{\dotc}

\multips(3.4,3.6)(0.2,0){2}{\dotd}

\multips(3.8,3.6)(0.2,0){1}{\dotc}

\multips(4.,3.6)(0.2,0){1}{\dote}

\multips(4.2,3.6)(0.2,0){3}{\dotd}

\multips(4.8,3.6)(0.2,0){2}{\dotc}

\multips(5.2,3.6)(0.2,0){3}{\dote}

\multips(5.8,3.6)(0.2,0){8}{\dotd}

\multips(1.2,3.7)(0.2,0){5}{\dota}

\multips(2.2,3.7)(0.2,0){2}{\dotb}

\multips(2.6,3.7)(0.2,0){1}{\dotd}

\multips(2.8,3.7)(0.2,0){3}{\dotc}

\multips(3.4,3.7)(0.2,0){2}{\dotd}

\multips(3.8,3.7)(0.2,0){1}{\dota}

\multips(4.,3.7)(0.2,0){1}{\dote}

\multips(4.2,3.7)(0.2,0){3}{\dotd}

\multips(4.8,3.7)(0.2,0){2}{\dotc}

\multips(5.2,3.7)(0.2,0){3}{\dote}

\multips(5.8,3.7)(0.2,0){9}{\dotd}

\multips(1.2,3.8)(0.2,0){7}{\dota}

\multips(2.6,3.8)(0.2,0){1}{\dotc}

\multips(2.8,3.8)(0.2,0){1}{\dotd}

\multips(3.,3.8)(0.2,0){1}{\dotc}

\multips(3.2,3.8)(0.2,0){1}{\dotd}

\multips(3.4,3.8)(0.2,0){1}{\dotc}

\multips(3.6,3.8)(0.2,0){1}{\dotb}

\multips(3.8,3.8)(0.2,0){1}{\dotd}

\multips(4.,3.8)(0.2,0){1}{\dotc}

\multips(4.2,3.8)(0.2,0){1}{\dote}

\multips(4.4,3.8)(0.2,0){4}{\dotd}

\multips(5.2,3.8)(0.2,0){2}{\dotc}

\multips(5.6,3.8)(0.2,0){2}{\dote}

\multips(6.,3.8)(0.2,0){9}{\dotd}

\multips(1.2,3.9)(0.2,0){7}{\dota}

\multips(2.6,3.9)(0.2,0){1}{\dotc}

\multips(2.8,3.9)(0.2,0){1}{\dotd}

\multips(3.,3.9)(0.2,0){1}{\dota}

\multips(3.2,3.9)(0.2,0){1}{\dotd}

\multips(3.4,3.9)(0.2,0){1}{\dotc}

\multips(3.6,3.9)(0.2,0){1}{\dote}

\multips(3.8,3.9)(0.2,0){1}{\dotd}

\multips(4.,3.9)(0.2,0){1}{\dotc}

\multips(4.2,3.9)(0.2,0){1}{\dote}

\multips(4.4,3.9)(0.2,0){1}{\dotb}

\multips(4.6,3.9)(0.2,0){3}{\dotd}

\multips(5.2,3.9)(0.2,0){1}{\dotc}

\multips(5.4,3.9)(0.2,0){3}{\dote}

\multips(6.,3.9)(0.2,0){1}{\dotb}

\multips(6.2,3.9)(0.2,0){9}{\dotd}

\multips(1.2,4.)(0.2,0){5}{\dota}

\multips(2.2,4.)(0.2,0){2}{\dotb}

\multips(2.6,4.)(0.2,0){1}{\dota}

\multips(2.8,4.)(0.2,0){1}{\dotb}

\multips(3.,4.)(0.2,0){2}{\dotd}

\multips(3.4,4.)(0.2,0){2}{\dotc}

\multips(3.8,4.)(0.2,0){1}{\dotd}

\multips(4.,4.)(0.2,0){1}{\dota}

\multips(4.2,4.)(0.2,0){1}{\dotc}

\multips(4.4,4.)(0.2,0){1}{\dote}

\multips(4.6,4.)(0.2,0){3}{\dotd}

\multips(5.2,4.)(0.2,0){1}{\dota}

\multips(5.4,4.)(0.2,0){1}{\dotc}

\multips(5.6,4.)(0.2,0){4}{\dote}

\multips(6.4,4.)(0.2,0){9}{\dotd}

\multips(1.2,4.1)(0.2,0){7}{\dota}

\multips(2.6,4.1)(0.2,0){1}{\dotd}

\multips(2.8,4.1)(0.2,0){1}{\dotc}

\multips(3.,4.1)(0.2,0){1}{\dotd}

\multips(3.2,4.1)(0.2,0){1}{\dotb}

\multips(3.4,4.1)(0.2,0){1}{\dota}

\multips(3.6,4.1)(0.2,0){1}{\dotc}

\multips(3.8,4.1)(0.2,0){1}{\dotb}

\multips(4.,4.1)(0.2,0){1}{\dotd}

\multips(4.2,4.1)(0.2,0){2}{\dote}

\multips(4.6,4.1)(0.2,0){4}{\dotd}

\multips(5.4,4.1)(0.2,0){2}{\dotc}

\multips(5.8,4.1)(0.2,0){4}{\dote}

\multips(6.6,4.1)(0.2,0){1}{\dotb}

\multips(6.8,4.1)(0.2,0){8}{\dotd}

\multips(1.2,4.2)(0.2,0){7}{\dota}

\multips(2.6,4.2)(0.2,0){1}{\dotb}

\multips(2.8,4.2)(0.2,0){1}{\dotc}

\multips(3.,4.2)(0.2,0){1}{\dotb}

\multips(3.2,4.2)(0.2,0){1}{\dote}

\multips(3.4,4.2)(0.2,0){1}{\dotd}

\multips(3.6,4.2)(0.2,0){1}{\dotc}

\multips(3.8,4.2)(0.2,0){1}{\dote}

\multips(4.,4.2)(0.2,0){1}{\dotd}

\multips(4.2,4.2)(0.2,0){3}{\dotc}

\multips(4.8,4.2)(0.2,0){4}{\dotd}

\multips(5.6,4.2)(0.2,0){6}{\dote}

\multips(6.8,4.2)(0.2,0){9}{\dotd}

\multips(1.2,4.3)(0.2,0){5}{\dota}

\multips(2.2,4.3)(0.2,0){2}{\dotb}

\multips(2.6,4.3)(0.2,0){1}{\dotc}

\multips(2.8,4.3)(0.2,0){1}{\dota}

\multips(3.,4.3)(0.2,0){2}{\dotc}

\multips(3.4,4.3)(0.2,0){2}{\dotd}

\multips(3.8,4.3)(0.2,0){1}{\dote}

\multips(4.,4.3)(0.2,0){1}{\dotd}

\multips(4.2,4.3)(0.2,0){2}{\dotc}

\multips(4.6,4.3)(0.2,0){1}{\dote}

\multips(4.8,4.3)(0.2,0){4}{\dotd}

\multips(5.6,4.3)(0.2,0){1}{\dotc}

\multips(5.8,4.3)(0.2,0){6}{\dote}

\multips(7.,4.3)(0.2,0){9}{\dotd}

\multips(1.2,4.4)(0.2,0){6}{\dota}

\multips(2.4,4.4)(0.2,0){1}{\dotc}

\multips(2.6,4.4)(0.2,0){1}{\dota}

\multips(2.8,4.4)(0.2,0){1}{\dotd}

\multips(3.,4.4)(0.2,0){2}{\dotc}

\multips(3.4,4.4)(0.2,0){2}{\dotd}

\multips(3.8,4.4)(0.2,0){1}{\dotc}

\multips(4.,4.4)(0.2,0){2}{\dotd}

\multips(4.4,4.4)(0.2,0){1}{\dotc}

\multips(4.6,4.4)(0.2,0){2}{\dote}

\multips(5.,4.4)(0.2,0){4}{\dotd}

\multips(5.8,4.4)(0.2,0){6}{\dote}

\multips(7.,4.4)(0.2,0){10}{\dotd}

\multips(1.2,4.5)(0.2,0){8}{\dota}

\multips(2.8,4.5)(0.2,0){1}{\dotd}

\multips(3.,4.5)(0.2,0){2}{\dota}

\multips(3.4,4.5)(0.2,0){1}{\dotb}

\multips(3.6,4.5)(0.2,0){1}{\dotd}

\multips(3.8,4.5)(0.2,0){1}{\dotc}

\multips(4.,4.5)(0.2,0){2}{\dotd}

\multips(4.4,4.5)(0.2,0){1}{\dota}

\multips(4.6,4.5)(0.2,0){2}{\dote}

\multips(5.,4.5)(0.2,0){1}{\dotb}

\multips(5.2,4.5)(0.2,0){3}{\dotd}

\multips(5.8,4.5)(0.2,0){2}{\dotc}

\multips(6.2,4.5)(0.2,0){2}{\dote}

\multips(6.6,4.5)(0.2,0){1}{\dotc}

\multips(6.8,4.5)(0.2,0){1}{\dote}

\multips(7.,4.5)(0.2,0){11}{\dotd}

\multips(1.2,4.6)(0.2,0){5}{\dota}

\multips(2.2,4.6)(0.2,0){1}{\dotc}

\multips(2.4,4.6)(0.2,0){2}{\dotb}

\multips(2.8,4.6)(0.2,0){1}{\dotc}

\multips(3.,4.6)(0.2,0){2}{\dotd}

\multips(3.4,4.6)(0.2,0){1}{\dotc}

\multips(3.6,4.6)(0.2,0){1}{\dotb}

\multips(3.8,4.6)(0.2,0){1}{\dotd}

\multips(4.,4.6)(0.2,0){1}{\dote}

\multips(4.2,4.6)(0.2,0){2}{\dotd}

\multips(4.6,4.6)(0.2,0){1}{\dotc}

\multips(4.8,4.6)(0.2,0){2}{\dote}

\multips(5.2,4.6)(0.2,0){3}{\dotd}

\multips(5.8,4.6)(0.2,0){3}{\dotc}

\multips(6.4,4.6)(0.2,0){4}{\dote}

\multips(7.2,4.6)(0.2,0){11}{\dotd}

\multips(1.2,4.7)(0.2,0){6}{\dota}

\multips(2.4,4.7)(0.2,0){1}{\dotb}

\multips(2.6,4.7)(0.2,0){2}{\dotc}

\multips(3.,4.7)(0.2,0){2}{\dotd}

\multips(3.4,4.7)(0.2,0){2}{\dotc}

\multips(3.8,4.7)(0.2,0){1}{\dotd}

\multips(4.,4.7)(0.2,0){1}{\dotc}

\multips(4.2,4.7)(0.2,0){2}{\dotd}

\multips(4.6,4.7)(0.2,0){1}{\dotc}

\multips(4.8,4.7)(0.2,0){1}{\dote}

\multips(5.,4.7)(0.2,0){1}{\dotc}

\multips(5.2,4.7)(0.2,0){3}{\dotd}

\multips(5.8,4.7)(0.2,0){1}{\dota}

\multips(6.,4.7)(0.2,0){2}{\dotc}

\multips(6.4,4.7)(0.2,0){4}{\dote}

\multips(7.2,4.7)(0.2,0){12}{\dotd}

\multips(1.2,4.8)(0.2,0){6}{\dota}

\multips(2.4,4.8)(0.2,0){2}{\dotc}

\multips(2.8,4.8)(0.2,0){1}{\dota}

\multips(3.,4.8)(0.2,0){1}{\dotb}

\multips(3.2,4.8)(0.2,0){1}{\dotd}

\multips(3.4,4.8)(0.2,0){2}{\dotc}

\multips(3.8,4.8)(0.2,0){1}{\dotd}

\multips(4.,4.8)(0.2,0){1}{\dotc}

\multips(4.2,4.8)(0.2,0){2}{\dotd}

\multips(4.6,4.8)(0.2,0){2}{\dotc}

\multips(5.,4.8)(0.2,0){1}{\dote}

\multips(5.2,4.8)(0.2,0){1}{\dotc}

\multips(5.4,4.8)(0.2,0){3}{\dotd}

\multips(6.,4.8)(0.2,0){2}{\dotc}

\multips(6.4,4.8)(0.2,0){2}{\dote}

\multips(6.8,4.8)(0.2,0){1}{\dotc}

\multips(7.,4.8)(0.2,0){2}{\dote}

\multips(7.4,4.8)(0.2,0){12}{\dotd}

\multips(1.2,4.9)(0.2,0){9}{\dota}

\multips(3.,4.9)(0.2,0){1}{\dotc}

\multips(3.2,4.9)(0.2,0){1}{\dotb}

\multips(3.4,4.9)(0.2,0){2}{\dotc}

\multips(3.8,4.9)(0.2,0){1}{\dotd}

\multips(4.,4.9)(0.2,0){1}{\dotc}

\multips(4.2,4.9)(0.2,0){2}{\dotd}

\multips(4.6,4.9)(0.2,0){1}{\dota}

\multips(4.8,4.9)(0.2,0){1}{\dote}

\multips(5.,4.9)(0.2,0){1}{\dotc}

\multips(5.2,4.9)(0.2,0){1}{\dote}

\multips(5.4,4.9)(0.2,0){3}{\dotd}

\multips(6.,4.9)(0.2,0){2}{\dotc}

\multips(6.4,4.9)(0.2,0){1}{\dote}

\multips(6.6,4.9)(0.2,0){1}{\dotc}

\multips(6.8,4.9)(0.2,0){3}{\dote}

\multips(7.4,4.9)(0.2,0){13}{\dotd}

\multips(1.2,5.)(0.2,0){6}{\dota}

\multips(2.4,5.)(0.2,0){3}{\dotb}

\multips(3.,5.)(0.2,0){1}{\dota}

\multips(3.2,5.)(0.2,0){1}{\dotc}

\multips(3.4,5.)(0.2,0){1}{\dotd}

\multips(3.6,5.)(0.2,0){1}{\dota}

\multips(3.8,5.)(0.2,0){1}{\dotc}

\multips(4.,5.)(0.2,0){1}{\dota}

\multips(4.2,5.)(0.2,0){1}{\dote}

\multips(4.4,5.)(0.2,0){1}{\dotb}

\multips(4.6,5.)(0.2,0){1}{\dotd}

\multips(4.8,5.)(0.2,0){1}{\dota}

\multips(5.,5.)(0.2,0){1}{\dotc}

\multips(5.2,5.)(0.2,0){2}{\dote}

\multips(5.6,5.)(0.2,0){4}{\dotd}

\multips(6.4,5.)(0.2,0){4}{\dotc}

\multips(7.2,5.)(0.2,0){3}{\dote}

\multips(7.8,5.)(0.2,0){12}{\dotd}

\multips(1.2,5.1)(0.2,0){5}{\dota}

\multips(2.2,5.1)(0.2,0){1}{\dotb}

\multips(2.4,5.1)(0.2,0){2}{\dota}

\multips(2.8,5.1)(0.2,0){1}{\dotc}

\multips(3.,5.1)(0.2,0){2}{\dotd}

\multips(3.4,5.1)(0.2,0){1}{\dotb}

\multips(3.6,5.1)(0.2,0){1}{\dotd}

\multips(3.8,5.1)(0.2,0){1}{\dotc}

\multips(4.,5.1)(0.2,0){1}{\dotd}

\multips(4.2,5.1)(0.2,0){2}{\dotc}

\multips(4.6,5.1)(0.2,0){3}{\dotd}

\multips(5.2,5.1)(0.2,0){1}{\dotc}

\multips(5.4,5.1)(0.2,0){2}{\dote}

\multips(5.8,5.1)(0.2,0){5}{\dotd}

\multips(6.8,5.1)(0.2,0){1}{\dotc}

\multips(7.,5.1)(0.2,0){4}{\dote}

\multips(7.8,5.1)(0.2,0){1}{\dotb}

\multips(8.,5.1)(0.2,0){12}{\dotd}

\multips(1.2,5.2)(0.2,0){9}{\dota}

\multips(3.,5.2)(0.2,0){1}{\dotb}

\multips(3.2,5.2)(0.2,0){1}{\dotd}

\multips(3.4,5.2)(0.2,0){1}{\dote}

\multips(3.6,5.2)(0.2,0){1}{\dotd}

\multips(3.8,5.2)(0.2,0){1}{\dotc}

\multips(4.,5.2)(0.2,0){1}{\dotd}

\multips(4.2,5.2)(0.2,0){2}{\dotc}

\multips(4.6,5.2)(0.2,0){3}{\dotd}

\multips(5.2,5.2)(0.2,0){2}{\dotc}

\multips(5.6,5.2)(0.2,0){1}{\dote}

\multips(5.8,5.2)(0.2,0){6}{\dotd}

\multips(7.,5.2)(0.2,0){6}{\dote}

\multips(8.2,5.2)(0.2,0){12}{\dotd}

\multips(1.2,5.3)(0.2,0){6}{\dota}

\multips(2.4,5.3)(0.2,0){1}{\dotb}

\multips(2.6,5.3)(0.2,0){2}{\dota}

\multips(3.,5.3)(0.2,0){1}{\dotb}

\multips(3.2,5.3)(0.2,0){1}{\dotd}

\multips(3.4,5.3)(0.2,0){1}{\dotc}

\multips(3.6,5.3)(0.2,0){1}{\dotd}

\multips(3.8,5.3)(0.2,0){1}{\dotc}

\multips(4.,5.3)(0.2,0){1}{\dotd}

\multips(4.2,5.3)(0.2,0){2}{\dotc}

\multips(4.6,5.3)(0.2,0){3}{\dotd}

\multips(5.2,5.3)(0.2,0){1}{\dotc}

\multips(5.4,5.3)(0.2,0){2}{\dote}

\multips(5.8,5.3)(0.2,0){6}{\dotd}

\multips(7.,5.3)(0.2,0){1}{\dote}

\multips(7.2,5.3)(0.2,0){1}{\dotc}

\multips(7.4,5.3)(0.2,0){4}{\dote}

\multips(8.2,5.3)(0.2,0){13}{\dotd}

\multips(1.2,5.4)(0.2,0){6}{\dota}

\multips(2.4,5.4)(0.2,0){1}{\dotc}

\multips(2.6,5.4)(0.2,0){1}{\dotb}

\multips(2.8,5.4)(0.2,0){1}{\dotd}

\multips(3.,5.4)(0.2,0){1}{\dotc}

\multips(3.2,5.4)(0.2,0){1}{\dotd}

\multips(3.4,5.4)(0.2,0){1}{\dotc}

\multips(3.6,5.4)(0.2,0){1}{\dotb}

\multips(3.8,5.4)(0.2,0){1}{\dota}

\multips(4.,5.4)(0.2,0){1}{\dotd}

\multips(4.2,5.4)(0.2,0){1}{\dota}

\multips(4.4,5.4)(0.2,0){1}{\dotc}

\multips(4.6,5.4)(0.2,0){1}{\dotb}

\multips(4.8,5.4)(0.2,0){2}{\dotd}

\multips(5.2,5.4)(0.2,0){1}{\dotc}

\multips(5.4,5.4)(0.2,0){2}{\dote}

\multips(5.8,5.4)(0.2,0){6}{\dotd}

\multips(7.,5.4)(0.2,0){6}{\dote}

\multips(8.2,5.4)(0.2,0){14}{\dotd}

\multips(1.2,5.5)(0.2,0){5}{\dota}

\multips(2.2,5.5)(0.2,0){1}{\dotb}

\multips(2.4,5.5)(0.2,0){1}{\dota}

\multips(2.6,5.5)(0.2,0){1}{\dotc}

\multips(2.8,5.5)(0.2,0){1}{\dotb}

\multips(3.,5.5)(0.2,0){4}{\dotc}

\multips(3.8,5.5)(0.2,0){1}{\dotd}

\multips(4.,5.5)(0.2,0){1}{\dote}

\multips(4.2,5.5)(0.2,0){1}{\dotd}

\multips(4.4,5.5)(0.2,0){1}{\dotc}

\multips(4.6,5.5)(0.2,0){1}{\dote}

\multips(4.8,5.5)(0.2,0){3}{\dotd}

\multips(5.4,5.5)(0.2,0){1}{\dotc}

\multips(5.6,5.5)(0.2,0){1}{\dote}

\multips(5.8,5.5)(0.2,0){6}{\dotd}

\multips(7.,5.5)(0.2,0){1}{\dotc}

\multips(7.2,5.5)(0.2,0){1}{\dote}

\multips(7.4,5.5)(0.2,0){1}{\dotc}

\multips(7.6,5.5)(0.2,0){4}{\dote}

\multips(8.4,5.5)(0.2,0){14}{\dotd}

\multips(1.2,5.6)(0.2,0){7}{\dota}

\multips(2.6,5.6)(0.2,0){1}{\dotc}

\multips(2.8,5.6)(0.2,0){1}{\dotb}

\multips(3.,5.6)(0.2,0){1}{\dota}

\multips(3.2,5.6)(0.2,0){1}{\dotc}

\multips(3.4,5.6)(0.2,0){1}{\dota}

\multips(3.6,5.6)(0.2,0){1}{\dotc}

\multips(3.8,5.6)(0.2,0){1}{\dotd}

\multips(4.,5.6)(0.2,0){1}{\dote}

\multips(4.2,5.6)(0.2,0){1}{\dotd}

\multips(4.4,5.6)(0.2,0){1}{\dota}

\multips(4.6,5.6)(0.2,0){1}{\dote}

\multips(4.8,5.6)(0.2,0){3}{\dotd}

\multips(5.4,5.6)(0.2,0){1}{\dotc}

\multips(5.6,5.6)(0.2,0){1}{\dote}

\multips(5.8,5.6)(0.2,0){1}{\dotb}

\multips(6.,5.6)(0.2,0){5}{\dotd}

\multips(7.,5.6)(0.2,0){1}{\dotc}

\multips(7.2,5.6)(0.2,0){6}{\dote}

\multips(8.4,5.6)(0.2,0){15}{\dotd}

\multips(1.2,5.7)(0.2,0){6}{\dota}

\multips(2.4,5.7)(0.2,0){1}{\dotc}

\multips(2.6,5.7)(0.2,0){1}{\dota}

\multips(2.8,5.7)(0.2,0){1}{\dotc}

\multips(3.,5.7)(0.2,0){1}{\dotd}

\multips(3.2,5.7)(0.2,0){1}{\dotc}

\multips(3.4,5.7)(0.2,0){1}{\dotd}

\multips(3.6,5.7)(0.2,0){1}{\dotc}

\multips(3.8,5.7)(0.2,0){1}{\dotd}

\multips(4.,5.7)(0.2,0){1}{\dotc}

\multips(4.2,5.7)(0.2,0){2}{\dotd}

\multips(4.6,5.7)(0.2,0){1}{\dotc}

\multips(4.8,5.7)(0.2,0){1}{\dote}

\multips(5.,5.7)(0.2,0){3}{\dotd}

\multips(5.6,5.7)(0.2,0){1}{\dotc}

\multips(5.8,5.7)(0.2,0){1}{\dote}

\multips(6.,5.7)(0.2,0){1}{\dotb}

\multips(6.2,5.7)(0.2,0){4}{\dotd}

\multips(7.,5.7)(0.2,0){3}{\dotc}

\multips(7.6,5.7)(0.2,0){2}{\dote}

\multips(8.,5.7)(0.2,0){1}{\dotc}

\multips(8.2,5.7)(0.2,0){3}{\dote}

\multips(8.8,5.7)(0.2,0){14}{\dotd}

\multips(1.2,5.8)(0.2,0){7}{\dota}

\multips(2.6,5.8)(0.2,0){1}{\dotb}

\multips(2.8,5.8)(0.2,0){1}{\dota}

\multips(3.,5.8)(0.2,0){1}{\dotb}

\multips(3.2,5.8)(0.2,0){1}{\dota}

\multips(3.4,5.8)(0.2,0){1}{\dotd}

\multips(3.6,5.8)(0.2,0){1}{\dota}

\multips(3.8,5.8)(0.2,0){2}{\dotc}

\multips(4.2,5.8)(0.2,0){2}{\dotd}

\multips(4.6,5.8)(0.2,0){1}{\dotc}

\multips(4.8,5.8)(0.2,0){2}{\dote}

\multips(5.2,5.8)(0.2,0){2}{\dotd}

\multips(5.6,5.8)(0.2,0){2}{\dotc}

\multips(6.,5.8)(0.2,0){1}{\dote}

\multips(6.2,5.8)(0.2,0){1}{\dotc}

\multips(6.4,5.8)(0.2,0){4}{\dotd}

\multips(7.2,5.8)(0.2,0){1}{\dotc}

\multips(7.4,5.8)(0.2,0){7}{\dote}

\multips(8.8,5.8)(0.2,0){15}{\dotd}

\multips(1.2,5.9)(0.2,0){7}{\dota}

\multips(2.6,5.9)(0.2,0){1}{\dotc}

\multips(2.8,5.9)(0.2,0){1}{\dotd}

\multips(3.,5.9)(0.2,0){1}{\dotc}

\multips(3.2,5.9)(0.2,0){1}{\dotd}

\multips(3.4,5.9)(0.2,0){1}{\dotb}

\multips(3.6,5.9)(0.2,0){1}{\dotd}

\multips(3.8,5.9)(0.2,0){1}{\dotc}

\multips(4.,5.9)(0.2,0){1}{\dota}

\multips(4.2,5.9)(0.2,0){1}{\dotb}

\multips(4.4,5.9)(0.2,0){1}{\dotd}

\multips(4.6,5.9)(0.2,0){1}{\dotc}

\multips(4.8,5.9)(0.2,0){2}{\dote}

\multips(5.2,5.9)(0.2,0){2}{\dotd}

\multips(5.6,5.9)(0.2,0){1}{\dotc}

\multips(5.8,5.9)(0.2,0){3}{\dote}

\multips(6.4,5.9)(0.2,0){5}{\dotd}

\multips(7.4,5.9)(0.2,0){1}{\dotc}

\multips(7.6,5.9)(0.2,0){7}{\dote}

\multips(9.,5.9)(0.2,0){15}{\dotd}

\multips(1.2,6.)(0.2,0){6}{\dota}

\multips(2.4,6.)(0.2,0){1}{\dotc}

\multips(2.6,6.)(0.2,0){1}{\dota}

\multips(2.8,6.)(0.2,0){1}{\dotb}

\multips(3.,6.)(0.2,0){1}{\dotc}

\multips(3.2,6.)(0.2,0){1}{\dotd}

\multips(3.4,6.)(0.2,0){1}{\dotc}

\multips(3.6,6.)(0.2,0){1}{\dotd}

\multips(3.8,6.)(0.2,0){1}{\dotc}

\multips(4.,6.)(0.2,0){1}{\dotd}

\multips(4.2,6.)(0.2,0){1}{\dote}

\multips(4.4,6.)(0.2,0){2}{\dotd}

\multips(4.8,6.)(0.2,0){1}{\dotc}

\multips(5.,6.)(0.2,0){1}{\dote}

\multips(5.2,6.)(0.2,0){3}{\dotd}

\multips(5.8,6.)(0.2,0){1}{\dote}

\multips(6.,6.)(0.2,0){1}{\dotc}

\multips(6.2,6.)(0.2,0){2}{\dote}

\multips(6.6,6.)(0.2,0){4}{\dotd}

\multips(7.4,6.)(0.2,0){2}{\dotc}

\multips(7.8,6.)(0.2,0){7}{\dote}

\multips(9.2,6.)(0.2,0){15}{\dotd}

\multips(1.2,6.1)(0.2,0){8}{\dota}

\multips(2.8,6.1)(0.2,0){1}{\dotb}

\multips(3.,6.1)(0.2,0){1}{\dota}

\multips(3.2,6.1)(0.2,0){1}{\dotb}

\multips(3.4,6.1)(0.2,0){1}{\dotc}

\multips(3.6,6.1)(0.2,0){1}{\dotd}

\multips(3.8,6.1)(0.2,0){1}{\dotc}

\multips(4.,6.1)(0.2,0){1}{\dotd}

\multips(4.2,6.1)(0.2,0){1}{\dote}

\multips(4.4,6.1)(0.2,0){1}{\dotb}

\multips(4.6,6.1)(0.2,0){1}{\dotd}

\multips(4.8,6.1)(0.2,0){2}{\dotc}

\multips(5.2,6.1)(0.2,0){3}{\dotd}

\multips(5.8,6.1)(0.2,0){1}{\dotc}

\multips(6.,6.1)(0.2,0){3}{\dote}

\multips(6.6,6.1)(0.2,0){1}{\dotb}

\multips(6.8,6.1)(0.2,0){3}{\dotd}

\multips(7.4,6.1)(0.2,0){2}{\dotc}

\multips(7.8,6.1)(0.2,0){1}{\dote}

\multips(8.,6.1)(0.2,0){1}{\dotc}

\multips(8.2,6.1)(0.2,0){1}{\dote}

\multips(8.4,6.1)(0.2,0){1}{\dotc}

\multips(8.6,6.1)(0.2,0){3}{\dote}

\multips(9.2,6.1)(0.2,0){1}{\dotb}

\multips(9.4,6.1)(0.2,0){15}{\dotd}

\multips(1.2,6.2)(0.2,0){7}{\dota}

\multips(2.6,6.2)(0.2,0){1}{\dotb}

\multips(2.8,6.2)(0.2,0){1}{\dotc}

\multips(3.,6.2)(0.2,0){1}{\dota}

\multips(3.2,6.2)(0.2,0){3}{\dotc}

\multips(3.8,6.2)(0.2,0){1}{\dota}

\multips(4.,6.2)(0.2,0){1}{\dotd}

\multips(4.2,6.2)(0.2,0){1}{\dotc}

\multips(4.4,6.2)(0.2,0){1}{\dote}

\multips(4.6,6.2)(0.2,0){1}{\dotd}

\multips(4.8,6.2)(0.2,0){2}{\dotc}

\multips(5.2,6.2)(0.2,0){1}{\dote}

\multips(5.4,6.2)(0.2,0){2}{\dotd}

\multips(5.8,6.2)(0.2,0){2}{\dotc}

\multips(6.2,6.2)(0.2,0){3}{\dote}

\multips(6.8,6.2)(0.2,0){4}{\dotd}

\multips(7.6,6.2)(0.2,0){1}{\dota}

\multips(7.8,6.2)(0.2,0){8}{\dote}

\multips(9.4,6.2)(0.2,0){16}{\dotd}

\multips(1.2,6.3)(0.2,0){7}{\dota}

\multips(2.6,6.3)(0.2,0){1}{\dotb}

\multips(2.8,6.3)(0.2,0){1}{\dota}

\multips(3.,6.3)(0.2,0){1}{\dotd}

\multips(3.2,6.3)(0.2,0){1}{\dotc}

\multips(3.4,6.3)(0.2,0){1}{\dota}

\multips(3.6,6.3)(0.2,0){1}{\dotc}

\multips(3.8,6.3)(0.2,0){1}{\dota}

\multips(4.,6.3)(0.2,0){1}{\dotd}

\multips(4.2,6.3)(0.2,0){2}{\dotc}

\multips(4.6,6.3)(0.2,0){1}{\dotd}

\multips(4.8,6.3)(0.2,0){1}{\dota}

\multips(5.,6.3)(0.2,0){1}{\dotc}

\multips(5.2,6.3)(0.2,0){1}{\dote}

\multips(5.4,6.3)(0.2,0){2}{\dotd}

\multips(5.8,6.3)(0.2,0){2}{\dotc}

\multips(6.2,6.3)(0.2,0){3}{\dote}

\multips(6.8,6.3)(0.2,0){1}{\dotb}

\multips(7.,6.3)(0.2,0){4}{\dotd}

\multips(7.8,6.3)(0.2,0){1}{\dotc}

\multips(8.,6.3)(0.2,0){7}{\dote}

\multips(9.4,6.3)(0.2,0){17}{\dotd}

\multips(1.2,6.4)(0.2,0){7}{\dota}

\multips(2.6,6.4)(0.2,0){1}{\dotc}

\multips(2.8,6.4)(0.2,0){1}{\dota}

\multips(3.,6.4)(0.2,0){1}{\dotd}

\multips(3.2,6.4)(0.2,0){1}{\dotc}

\multips(3.4,6.4)(0.2,0){1}{\dota}

\multips(3.6,6.4)(0.2,0){1}{\dotc}

\multips(3.8,6.4)(0.2,0){2}{\dotd}

\multips(4.2,6.4)(0.2,0){2}{\dotc}

\multips(4.6,6.4)(0.2,0){1}{\dotd}

\multips(4.8,6.4)(0.2,0){1}{\dota}

\multips(5.,6.4)(0.2,0){1}{\dotc}

\multips(5.2,6.4)(0.2,0){1}{\dote}

\multips(5.4,6.4)(0.2,0){2}{\dotd}

\multips(5.8,6.4)(0.2,0){1}{\dotc}

\multips(6.,6.4)(0.2,0){1}{\dote}

\multips(6.2,6.4)(0.2,0){1}{\dotc}

\multips(6.4,6.4)(0.2,0){3}{\dote}

\multips(7.,6.4)(0.2,0){4}{\dotd}

\multips(7.8,6.4)(0.2,0){2}{\dotc}

\multips(8.2,6.4)(0.2,0){6}{\dote}

\multips(9.4,6.4)(0.2,0){18}{\dotd}

\multips(1.2,6.5)(0.2,0){8}{\dota}

\multips(2.8,6.5)(0.2,0){1}{\dotd}

\multips(3.,6.5)(0.2,0){1}{\dotc}

\multips(3.2,6.5)(0.2,0){1}{\dota}

\multips(3.4,6.5)(0.2,0){1}{\dotd}

\multips(3.6,6.5)(0.2,0){1}{\dotc}

\multips(3.8,6.5)(0.2,0){1}{\dotd}

\multips(4.,6.5)(0.2,0){1}{\dote}

\multips(4.2,6.5)(0.2,0){1}{\dota}

\multips(4.4,6.5)(0.2,0){1}{\dotc}

\multips(4.6,6.5)(0.2,0){3}{\dotd}

\multips(5.2,6.5)(0.2,0){2}{\dote}

\multips(5.6,6.5)(0.2,0){2}{\dotd}

\multips(6.,6.5)(0.2,0){3}{\dotc}

\multips(6.6,6.5)(0.2,0){2}{\dote}

\multips(7.,6.5)(0.2,0){4}{\dotd}

\multips(7.8,6.5)(0.2,0){2}{\dotc}

\multips(8.2,6.5)(0.2,0){1}{\dote}

\multips(8.4,6.5)(0.2,0){1}{\dotc}

\multips(8.6,6.5)(0.2,0){6}{\dote}

\multips(9.8,6.5)(0.2,0){17}{\dotd}

\end{pspicture}
\caption{$2 < p < q < 300$ }
\end{center}
\end{figure}


With \eqref{gpqbound}, we may translate the bounds on $g$ into bounds on $\nu_{pq}$. We can do better in the case when $\kappa + \lambda \geq p$, where, by Theorem \ref{T1:Main}, we have $g=\ga$.
\begin{theorem}\label{T6:main}
For primes $3<p<q$, with $\kappa + \lambda \geq p$ and $q \neq 2p-1,  3p-2$, we have
$$\nu_{pq} = \ga - pq + 1.$$
\end{theorem}
\noindent Thus, $\nu_{pq}$ attains the lower bound of \eqref{gpqbound} in this case.  For a twin prime pair, $\nu_{pq}$ lies about halfway between the bounds of \eqref{gpqbound} (see Theorem~\ref{T:twingroup}).  It appears that the upper bound is not attained for any prime pair.

 The  Type I pairs not covered by  Theorem~\ref{T5:Main}  (white in the Figure 1) are those for which $p> \kappa + \lambda >p-\lambda$  (see Remark~\ref{R:white}, Section~\ref{S:repg1}).  We plan to treat these pairs in a future paper.  For now, we note that the formula for the Frobenius number $g_{pq}$ depends on the  number theoretic relationship between $q/p$ and $q'/p'$.   Making this dependence precise involves the continued fraction
\begin{equation}\label{contfrac}
\frac{q}{p}=q_1 + \frac{1}{\displaystyle q_2 + \frac{1}{\displaystyle \ddots  \frac{1}{\displaystyle q_{n-1} + \frac{1}{\displaystyle q_n}}}}.
\end{equation}
The condition $\kappa + \lambda \geq p$, appearing in Theorem \ref{T1:Main}(ii),  is equivalent to $q_1+1 \leq q'/p'$. It appears that the next case is
$$
q_1+\frac{1}{q_2} \leq \frac{q'}{p'} <q_1+1
$$
and that an exact, though more complicated, formula for $g$ is also possible in this case. It seems likely that $g_{pq}$ depends on where $q'/p'$ lies in relation to the convergents of \eqref{contfrac}.

 \section{The motivating problems}\label{S:motivation}

     If a compact Riemann surface $X$ admits a finite group $G$ of  conformal  {\it automorphisms}, the quotient space $Y=X/G$ is itself a compact Riemann surface, and the quotient map $\Phi:X \rightarrow Y$ is  a holomorphic {\it branched covering map\,} of degree $n=|G|$ (the order of $G$).  This means that $\Phi$  is generically $n$-to-$1$ (or $n$-fold), but there is a finite subset $B \subset Y$, called the branch set,  over which the fibers have cardinality strictly less than $n$.   The {\it Riemann-Hurwitz relation}, a linear Diophantine equation, relates the topological data associated with  $\Phi$, namely, the genera of the surfaces, the degree of the covering, and the cardinalities of the fibers over the branch set.  It is a generalization of the multiplicative relation  between the Euler characteristics of the surfaces, $\chi(X) = n\cdot \chi(Y)$ which holds for $n$-fold {\it unbranched} covering maps.  (See \cite{FK}, Sections I.1 and I.2, for a fuller treatment of these ideas.)

  Branched covering maps need not arise as quotient maps of group actions.  Those that do must satisfy an extra regularity condition:  for every $y \in B$, there exists a divisor $n_y >1$ of $n$ such that the fiber over $y$ consists of precisely $n/n_y$ points, at  which the $n$ sheets of the covering come together in sets of $n_y$.   The integers $n_y$, $y \in B$, are called the {\it branching indices}, and the covering is called {\it semi-regular}.    When $\Phi: X \rightarrow Y$ is a semi-regular branched covering, the Riemann-Hurwitz relation is
  \begin{equation}\label{E:RH1}
  2(\gamma -1) =  2 n (\eta -1) + n \sum_{y\in B}\left(1-\frac{1}{n_y}\right),
  \end{equation}
  where $\gamma,\eta$, are the genera of $X$, $Y$, respectively, $n$ is the degree of the covering, and  $n_y$, $y \in B$ are the branching indices.  If $\Phi$ can be realized as the quotient map of a group action, the covering is called {\it regular}.

  We now specialize to the case where $n$, the degree of the covering,  is  a square-free odd integer with $s\geq 1$ distinct prime factors $p_i$, $i=1,2, \dots s$.  The  $2^s$ divisors of $n$ are  in one-to-one correspondence with the set  ${\cal B}$ of binary bit strings of length $s$.  Let $I$ denote a bit string of length $s$, and $\overline 0$, $\overline 1$  the bit strings consisting of all $0$'s, and all $1$'s,  respectively.  Let $n_I$ denote the divisor of $n$ associated with the bit string $I$, so that, for example, $n_{\overline 0}=1$ and $n_{\overline 1}=n$, and, more generally, $p_i$ is a factor of $n_I$ if and only if the $i$th bit of $I$ is $1$.  Then \eqref{E:RH1} implies the  Riemann-Hurwitz formula for an $n$-fold semi-regular branched covering is
  \begin{equation}\label{E:RHbinary}
  \gamma+n-1 = x_{\overline 0} n +  \sum_{I \ne \overline 0}x_I\frac{n(n_I - 1)}{2n_I},
  \end{equation}
 where $x_I$ ($I \neq \overline 0$)  is the number of points in the branch set with branching index $n_I$, and $x_{\overline 0} = \eta$, the genus of $Y$.  The integers
 $$d_{\overline 0}=n, \qquad d_I = \frac{n(n_I -1)}{2n_I},\quad  I \in {\cal B}, I \ne {\overline 0} $$
 have gcd $=1$, so there is a $2^s$-dimensional Frobenius number $g(\{d_I: I \in {\cal B}\})$.      By the general theory of branched coverings, there is surface $X$ of genus $\gamma$ which is an $n$-fold semi-regular covering  if and only if there is a $2^s$-tuple $(x_I)_{ I \in \cal B}$ of nonnegative integers satisfying \eqref{E:RHbinary}.   It follows that there is a largest {\it non-genus} of a  semi-regular $n$-fold covering, namely, the additive translate $-n+1+ g(\{d_I: I \in {\cal B}\})$ of the $2^s$-dimensional Frobenius number $g(\{d_I: I \in {\cal B}\})$.

 \vskip 3mm
 \noindent
 {\bf Problem I:} For every square-free odd $n$ with $s$ distinct prime factors, determine the largest non-genus of a semi-regular $n$-fold covering. This genus is $g_n - n+1$ where $g_n$ is the $2^s$-dimensional Frobenius number $g(\{d_I: I \in {\cal B}\})$.

 \vskip 3mm

   The case $s=1$ of Problem I follows immediately  from the previous paragraph and Sylvester's formula \eqref{E:sylvester} for the $2$-dimensional Frobenius number.
    \begin{prop}\label{P:p}  Let $p$ be an odd prime.   The largest non-genus of a  $p$-fold semi-regular branched covering is $p'(p-3) -p$, the additive translate $g(\{p,p'\})-p+1$ of the two dimensional Frobenius number $g(\{p, p'\})$.
             \end{prop}
 \noindent Note that this integer is $<0$ for $p=3, 5$, so that there is a semi-regular $3$- or $5$- fold branched covering of every genus.

 \subsection{Group actions}\label{ss:groups}  We now give a set of necessary and sufficient conditions  for the existence of a {\it regular} {\it cyclic} $n$-fold  branched covering $\Phi: X \rightarrow Y$,   that is,  a covering realizable as the quotient map of a cyclic group ${\mathbb Z}_n$ of automorphisms acting on the compact Riemann surface $X$.  The conditions are a special case of a more general set of conditions for the existence of an action by an arbitrary finite group $G$ of order $n$.  The necessary and sufficient condition in the general case is the existence of a {\it partial monodromy presentation} of $G$, having a  form  dictated by the genus of $Y$ and the branching indices.  If the genus of $Y$ is $\eta$, and the branching indices are $r_1, r_2, \dots, r_k$, the monodromy presentation of $G$ must have $2\eta + k$ generators $a_1, b_1, \dots a_\eta, b_\eta, c_1, \dots, c_k,$   where $c_i$ has order $r_i$ and, among other possible relations,
  \begin{equation}\label{E:longrelation}
  \prod_{i=1}^{\eta}[a_i,b_i]\prod_{j=1}^k c_j = 1,
  \end{equation}
where $[a_i,b_i]$ denotes the commutator and $1$ denotes the identity element in $G$.
(For a fuller explanation of the general case, see, e.g., \cite{B01}, or  \cite[Chapter III, Section 3]{RM91},  or \cite[Section 1.7]{G77}.)

  \begin{lemma}\label{L:admissibility} Let $n=p_1\dots p_s$, $s \geq 1$,  a square-free odd integer with prime factors $p_1, \dots, p_s$.  Let $\gamma \geq 0$.  Let $(x_I \geq 0)_{I \in {\cal B}}$ be a $2^s$-tuple satisfying \eqref{E:RHbinary} for $n$, $\gamma$.  There is a compact Riemann surface of genus $\gamma$ admitting a group of automorphisms ${\mathbb Z}_n$ such that quotient surface has genus $x_{\overline 0}$ and $x_I$ points of branching indices $n_I$, $I \neq \overline 0$, if and only if  the tuple $(x_I \geq 0)_{I \in {\cal B}}$ satisfies the admissibility conditions:
  \begin{align}
  \sum_{I \in {\cal B}^i} x_I &\neq 1, \quad i=1,2, \dots, s; \label{E:bincond1}\\
  x_{\overline 0} + \sum_{I \in {\cal B}^i}x_I & \neq 0, \quad i=1,2, \dots, s, \label{E:bincond2}
  \end{align}
 where ${\cal B}^i \subset {\cal B}$ is the set of bit strings of length $s$ whose $i$th bit is $1$.
  \end{lemma}
 \begin{proof}   A partial monodromy presentation of ${\mathbb Z}_n$ dictated by the tuple $(x_I \geq 0)_{I \in {\cal B}}$ would have $2x_{\overline 0}$ generators $a_0, b_0, \dots, a_{x_{\overline 0}}, b_{x_{\overline 0}}$ of unspecified order, and
 $x_I$ generators $c_I$ of order $n_I$ for each $I \in {\cal B}$, $I \neq \overline 0$.  Since all commutators are trivial in an abelian group, the elements $a_j, b_j$ can be omitted from the relation \eqref{E:longrelation}.    If the $i$th condition in \eqref{E:bincond1} fails, the group product on the left-hand side of \eqref{E:longrelation} would contain exactly one  element of order divisible by $p_i$, and hence could not be equal to the identity. If the $i$th condition in \eqref{E:bincond2} fails,  the generating set would contain no elements of order divisible by $p_i$,  a contradiction.   This proves the necessity of the conditions.  To prove sufficiency of the conditions, one verifies that a partial monodromy presentation of ${\mathbb Z}_{n}$ can be constructed in all other cases; this is left as an exercise.
\end{proof}

If there exist tuples $(x_I \geq 0)_{I \in {\cal B}}$ satisfying \eqref{E:RHbinary} for some $\gamma$, $n$,  but none of them satisfy  all the admissibility conditions in \eqref{E:bincond1} and \eqref{E:bincond2}, then $\gamma$ is the genus of an $n$-fold semi-regular covering, but a {\it non-genus} for a ${\mathbb Z}_n$ action.    There exists a largest non-genus of a ${\mathbb Z}_n$ action \cite{K87}, and it must be at least as large as the largest non-genus of an $n$-fold semi-regular covering.

\vskip 3mm
 \noindent
 {\bf Problem II:} For every square-free odd $n$, determine the largest non-genus $\nu_n$ of ${\mathbb Z}_n$.

 \vskip 3mm

\begin{theorem} \label{gnnn} Let $n$ be a square-free odd integer with $s\geq 1$ distinct prime factors.  With $\nu_n$ denoting the largest non-genus of a ${\mathbb Z}_n$ action, and  $g_n$  the Frobenius number $g(\{d_I: I \in {\cal B}\})$, we have
\begin{equation}\label{E:groupineq}
g_n -n+1 \leq \nu_n \leq g_n.
\end{equation}
\end{theorem}
 \begin{proof}  The left-hand inequality is clear: $\gamma =g_n - n+1$ is the largest integer such no $2^s$-tuple  $(x_I)_{I \in{\cal B}}$ (admissible or not) satisfies \eqref{E:RHbinary}.  Hence $\nu_n$ must be at least as large as $g_n-n+1$.  For the right-hand inequality, let $(x_I)_{I \in {\cal B}}$ be a nonnegative $2^s$-tuple satisfying  \eqref{E:RHbinary} for some $\gamma$.   The tuple obtained from $(x_I)$ by replacing the final coordinate $x_{\overline 1}$ with $x_{\overline 1} + 2$ satisfies \eqref{E:RHbinary} with $\gamma$ replaced by $\gamma + n-1$.  Moreover, the new tuple satisfies the admissibility conditions \eqref{E:bincond1} and \eqref{E:bincond2}.  Thus, if there is a surface of genus $\gamma$ which is an $n$-fold semi-regular covering, there is a surface of genus $\gamma + n-1$ which admits a ${\mathbb Z}_n$ action.  Consequently, $\nu_n$ is no larger than $g_n-n+1 + n-1 = g_n$.
 \end{proof}
  In the case $s=1$, the Riemann-Hurwtiz relation is
\begin{equation}\label{E:RHs=1}
\gamma + p-1 = x_0p + x_1p'
\end{equation}
 and the admissibility conditions are  simply $x_0 + x_1 \neq 0$ and $x_1 \neq 1$.   It is easy to verify that  there is just one solution of \eqref{E:RHs=1} when  $\gamma =g(\{p,p'\})$, namely,  the inadmissible pair $(x_0, x_1) = (p'-1,1)$.  Thus the largest non-genus of a ${\mathbb Z}_p$ action  is  strictly greater than the largest non-genus of a semi-regular $p$-fold covering (cf. Proposition~\ref{P:p}).  In fact it is known (\cite{KM91}) that $\nu_p =g(\{p,p'\}) = g_p$.  This shows that the upper bound in \eqref{E:groupineq} can be attained.  We conjecture that $s=1$ is the only case in which this occurs.    We shall show in Section~\ref{S:surfaces} that when $s=2$, the lower bound in \eqref{E:groupineq} is attained for infinitely many $n=pq$.

 A group of square-free order is either cyclic or metacyclic (see, e.g., \cite{Hall}, Theorem~9.4.3).  A metacyclic group has a normal cyclic subgroup with a cyclic factor group.  If $s=1$, the only possible group is ${\mathbb Z}_p$. If $s=2$, there is  a (nonabelian) metacyclic group (of order $pq$) if and only if $p$ is a divsor of $q-1$.  Such a group contains no elements of order $pq$,  hence the quotient map  has no branching indices equal to $pq$, and   the corresponding Frobenius problem is $3$-, not $4$-dimensional.  The admissibility conditions for a partial monodromy presentation are (naturally) different.  A formula for the largest non-genus of a metacyclic group action of order $pq$ is given by the second author in \cite{W01}.  In section \ref{S:surfaces} we give a formula for the largest non-genus of ${\mathbb Z}_{pq}$ which, given $p$, is valid for all but finitely many $q>p$.

 \vskip 3mm
Henceforth we treat  Problems I and II exclusively for $n$ a product of two distinct primes. Until the last section, we revert to the purely number theoretic question of determining the  $4$-dimensional Frobenius number $g=g(\{d_0, d_1, d_2, d_3\})$, with the $d_i$ as defined at \eqref{E:d's}.

\section{Representability of integers $>\ga$}\label{S:representability>g0}

To prove that a certain integer $m$ is the Frobenius number $g$, we need to establish  that (a) $m$ is not representable as $\fr(x,y,z,w)$ for any  quadruple $(x, y, z, w)$ of nonnegative integers; and (b)  all integers $>m$ are representable in this way.  For (b), it suffices to show that all integers in the closed interval $[m+1, m+ d_1]$ are representable, since if $\fr(x,y,z,w)$ is a nonnegative representation of $k \in [m+1, m+d_1]$, then  $ \fr(x,y+l, z,w)$ is a nonnegative representation of $k + ld_1$, for any $l \geq 0$.   In this and subsequent sections we apply this method to $m=\ga, \gb$ and $ \gc$, as they are defined at \eqref{E:keynums}.  Having applied the method to $\ga$, it will be possible to reuse much of the work in the treatment of $\gb$ and $\gc$.

For $x,y,z,w \in \mathbb Q$,  the equation $\fr(x,y,z,w)=0$ determines a three dimensional vector subspace of ${\mathbb Q}^4$, whose span is the hyperplane orthogonal to the
vector $(d_0, d_1, d_2, d_3)$.  It has an obvious basis consisting of the three vectors
$$(d_1, -d_0, 0,0), \quad (0, d_2, -d_1, 0), \quad (0,0, d_3, -d_2).$$
It is easy to show that
\begin{align}
e_0 & = (p', -p, 0, 0)  \label{E:trivequiv} \\
e_1 & = (p', 0, 1, -p)  \label{E:fact0p} \\
e_2 & = (q', 1, 0, -q), \label{E:fact0q}
\end{align}
is also a basis.  This basis is convenient  since (an exercise shows) if  there are integer quadruples $(x,y, z,w)$ and $(x',y',z',w')$ such that $\fr(x,y,z,w) = \fr(x',y',z',w')$, then the  vector $(x-x', y-y', z-z', w-w')$ is an {\it integer}  linear combination of $e_0$, $e_1$ and $e_2$.   Thus,  since $\fr$ is linear, $\fr(x,y,z,w) = \fr(x',y',z',w')$ if and only if $(x',y',z',w') = (x,y,z,w) +\alpha e_1 + \beta e_2 + \gamma e_3$, for some $\alpha,\beta,\gamma \in {\mathbb Z}$.

\begin{prop}\label{P:1} All integers $> \ga$ are representable.
\end{prop}
To prove this, we show that for each integer $n$ in the closed interval $[\ga + 1, \ga + d_1]$, a nonnegative quadruple $(x,y,z,w)$ exists such that $\fr(x,y,z,w) = n$.    We first construct  quadruples (possibly with negative entries) representing the integers in  $[\ga + 1, \ga + d_1]$ and then show that they can be altered, if necessary, by adding an integer linear combination of the vectors $e_0, e_1, e_2$,  so that they become {\it nonnegative} quadruples.  We will  make use of  the following easily verified facts:
\begin{align}
\fr(0, -1, 0, 1) & = q' \label{E:factq'}\\
\fr(0,0, -1, 1) & = p' \label{E:factp'}\\
\fr(1, 0, 0, -2) & = 1. \label{E:fact1}
\end{align}

We start by obtaining a nonnegative representation of $\ga+1$, using $\fr(e_1)=0$ and \eqref{E:fact1}:
\begin{equation}\label{E:g0+1}
\ga + 1 = \fr(p'-1, p-1, \kappa, -1) - \fr(p', 0, 1, -p) + \fr(1,0,0,-2) = \fr(0,p-1, \kappa-1, p-3).
\end{equation}
We proceed to show that
 $\ga + 1 + t$ has a nonnegative representation for all $t \in [0, d_1-1]$.

Let an integer $t \in [0, d_1-1]$ be represented with the division algorithm as
\begin{equation}\label{E:paramdef}
t= aq' + bp' + c,\quad \text{with} \quad a \geq 0 \text{\ maximal,}\quad b \geq 0,\quad 0 \leq c \leq p'-1.
\end{equation}
 The triple $(a,b,c)$ is uniquely determined by $t$ and conversely.
 \begin{lemma}\label{L:t}
 If $t \in [0, d_1-1]$ has the  representation \eqref{E:paramdef}, then
\begin{enumerate}
\item $a \leq p-1$;
\item $a = p-1 \implies b=0$;
\item $b \leq \kappa'$;
\item $b = \kappa' \implies c < \lambda' \implies p>3$;
\item $b \geq \kappa \implies a \leq p-2$.
\end{enumerate}
\end{lemma}
\begin{proof}
(i) and (v):   If $a\geq p$, or if $a=p-1$ and $b \geq \kappa$, then $t \geq p'q=d_1$, contrary to assumption.  (ii):  If $a=p-1$ and $t \leq p'q$, then $bp'+c \leq p'-1$, which implies $b=0$.  (iii): If $b>\kappa'$, $a$ is not maximal. (iv):  If $b=\kappa'$ and $c \geq \lambda'$, $a$ is not maximal.  When $p=3$, $\lambda'=0$ and hence $c<\lambda'$ is impossible.
\end{proof}

It follows from  \eqref{E:factq'}-\eqref{E:fact1} that, for $t \in [0, d_1-1]$,
 $$
\ga + 1 + t = \fr(0, p-1, \kappa -1, p-3) + a\cdot\fr(0,-1,0,1) + b \cdot \fr(0,0,-1,1) + c\cdot \fr(1,0,0,-2).
$$
Thus $\ga+1+t =\fr(x,\ y,\ z,\ w)$,   where
\begin{align}
x  &= c\label{E:x}\\
y  &= p-1-a\label{E:y}\\
z  &= \kappa - 1-b \label{E:z} \\
w &  = p-3-2c+a+b. \label{E:w}
\end{align}
By definition, $x\geq 0$.  By Lemma~\ref{L:t}(i), $y \geq 0$.  $w \geq 0$ because $c \leq p'-1$ is equivalent to
\begin{equation}\label{E:p-3-2c}
p-3-2c \geq 0.
\end{equation}
Thus $z$ is the only component of the quadruple $(x,y,z,w)$ which might  be negative (if $b \geq \kappa$).  If this is the case,
\begin{equation}\label{E:z<0}
 b = \kappa + s\quad \text{ for some}\quad  0 \leq s \leq \kappa'-\kappa.
\end{equation}
The upper bound on $s$ is a consequence of  Lemma~\ref{L:t}(iii).  We now show that there is always an integer linear combination of the vectors \eqref{E:fact0p} and \eqref{E:fact0q}, which, when added to the quadruple defined by \eqref{E:x} - \eqref{E:w}, yields a nonnegative quadruple.   The argument will be divided into three parts (Lemmas~\ref{L:A}, \ref{L:B} and \ref{L:C}), according to whether $s$ is, respectively, less than, equal to, or greater than $\kappa'-\kappa - 1$.

 For notational convenience, we define the quadruple
\begin{equation}\label{E:equiv}
e(u,v) \equiv (u-1)e_2 + (v+1)e_1,
\end{equation}
where  $e_1$, $e_2$  are the vectors  \eqref{E:fact0p} and \eqref{E:fact0q}, respectively, and $u,v \in {\Bbb Z}$.

\begin{lemma}\label{L:A}
 If  $s<\kappa'-\kappa -1$,
 \begin{enumerate}
 \item $\kappa + s -(s+1) p \geq 0$;
 \item $(x',y',z', w') =(x, y, z, w) + e(1, s)$ is a nonnegative quadruple.
 \end{enumerate}
\end{lemma}
\begin{proof}
From $q' \geq \kappa' p'$ we obtain  $q-1 \geq \kappa'(p-1)$ $\iff$ $\kappa'-1 \geq \kappa' p-q$   $= (\kappa'-\kappa)p +\kappa p -q$, and hence
\begin{equation*}\label{E:kappa'ineq}
\kappa' - 1  \geq (\kappa' - \kappa)p - \lambda.
\end{equation*}
  It follows from this that
\begin{equation*}\label{E:anyl}
\kappa' - 1 - l \geq (\kappa' - \kappa)p - \lambda - l, \quad\text{for}\quad l \geq 0.
\end{equation*}
 In particular, since $\lambda \leq p-1$,
\begin{equation}\label{E:l>1}
\kappa'-1 - l  \geq (\kappa' - \kappa)p - lp \quad \text{if}\quad l \geq 1.
\end{equation}
 Putting $l = \kappa'-\kappa -s -1\geq 1$ in \eqref{E:l>1}, we obtain (i).
To prove (ii), we have $x'>x \geq 0$, $y'=y\geq 0$, and $z'=-(s+1) +(s+1) = 0$.  We need only show that
\begin{equation}\label{E:w'}
w' = p-3-2c+a+b - (s+1)p
\end{equation}
is nonnegative.  Recalling that $b=\kappa + s$, and using \eqref{E:p-3-2c}, we obtain
\begin{equation*}
w'\geq \kappa + s - (s+1)p.
\end{equation*}
Thus $w'\geq 0$ is a consequence of (i).
\end{proof}

If $s=\kappa'-\kappa-1$, then $w \geq (s+1)p$ easily implies that the fourth coordinate of $(x,y,z,w) + e(1,s)$ is positive.  The following lemma treats the case $w<(s+1)p$, where the fourth coordinate of $(x,y,z,w) + e(1,s)$ is negative.

\begin{lemma}\label{L:B}
  If $s=\kappa'-\kappa-1$ and $w<(s+1)p$,
\begin{enumerate}
\item $c-\lambda' \geq 0$;
\item  $(x',y',z', w') =(x, y, z, w) + e(0, \kappa'-1)$ is a nonnegative quadruple.
\end{enumerate}
\end{lemma}
\begin{proof}
 $w= p-3-2c + a + \kappa + s < (s+1)p$ is equivalent to
\begin{align*}
2c &>\kappa -3 - s(p-1) \\
c &> (\kappa - 3)/2 - sp' \\
c &\geq (\kappa - 1)/2 - sp'.
\end{align*}
Putting $s=\kappa'-\kappa - 1$, we have
\begin{align*}
c &\geq (\kappa -1 + p-1 )/2 - (\kappa'-\kappa)p' \\
&\geq (\kappa + \lambda -1)/2 - (\kappa'-\kappa)p' \\
&=\lambda',
\end{align*}
where we have used $\lambda \leq p-1$ and   Lemma~\ref{L:lambda'}.  Thus (i) is proved.
$x'=c-\lambda' $ which is $\geq 0$ by (i).  $y'=p-1-a-1 \geq 0$ by Lemma \ref{L:t}(v).  $z'= -s-1+\kappa'=\kappa \geq 1$.  Finally, $w'=w+q-k'p = w-(\kappa'-\kappa)p +\lambda$.
  Since $w \geq b = \kappa + s =\kappa'-1$, and $\kappa'-1 \geq (\kappa'-\kappa)p-\lambda$ by
  \eqref{E:kappa'ineq},
  \begin{equation*}
w'=w-(\kappa' - \kappa)p + \lambda \geq (\kappa'-\kappa)p - \lambda - (\kappa'-\kappa)p + \lambda = 0.
\end{equation*}
  Thus (ii) is proved.
\end{proof}

\begin{lemma}\label{L:C}
 If $s = \kappa'-\kappa$,  $(x',y',z',w')=(x,y,z,w) + e(0, \kappa')$ is a nonnegative quadruple.
\end{lemma}
\begin{proof} By Lemma~\ref{L:t}(iv), $c< \lambda'$.  Since both $c$ and $\lambda'$ are $\leq p'-1$,
\begin{equation}\label{E:lambdaC}
1 \leq \lambda'-c \leq p'-1.
\end{equation}
We have
\begin{align}
x' &= c-q' +(\kappa' + 1)p' =p'-(\lambda' - c)   \label{E:x'C}\\
y' &=p-2-a  \label{E:y'C}\\
z'& =-(s+1)+\kappa'+1 = \kappa   \label{E:z'C}\\
w'&=p-3-2c+a+\kappa' +q-(\kappa'+1)p \notag\\
&=-3-2c+a + q -\kappa'(p-1) \notag \\
&=-2-2c+a + q-1-\kappa'(p-1) \notag \\
&=-2-2c+a+ 2(q'-\kappa'p') \notag \\
&=-2-2c+a+2\lambda'  \notag \\
&=2(\lambda'-c) -2 +a \label{E:w'C}
\end{align}
$x', w' \geq 0$ by \eqref{E:lambdaC}. $y'\geq 0$ by Lemma~\ref{L:t}(v). Finally $z'=\kappa\geq 1$.
\end{proof}

Lemmas~\ref{L:A}, \ref{L:B} and \ref{L:C} together constitute a proof of Proposition~\ref{P:1} which implies that $g \leq \ga$ for all prime pairs $2<p<q$.

\section{Representability of $\ga$}\label{S:repg0}

In this section we prove

\begin{prop}\label{P:2} $\ga$ is representable if and only if $\kappa + \lambda < p$.
\end{prop}

    Suppose $\kappa + \lambda < p$, and put $(x',y',z',w')=e(0, \kappa'-1)+(p'-1, p-1, \kappa, -1) $.  Since
\begin{equation}\label{E:mainequiv}
 e(0, \kappa'-1) = (-q'+\kappa' p', -1, \kappa', q-\kappa' p)=(-\lambda', -1, \kappa, \lambda)
 \end{equation}
(the last equality being a consequence of $\kappa'= \kappa$), it easily verified that $(x',y',z',w')$ is a nonnegative quadruple representing $\ga$.

To prove the necessity of the condition, we employ a  number theoretic lemma whose  proof is a simple exercise.

\begin{lemma}\label{L:rigid}
 Let $m,n$ be relatively prime integers. If $am + bn=a'm+b'n$ for any $a,a',b,b' \in \mathbb Z$ then there exists an integer $l$ such that $a'=a-ln$ and $b'=b+lm$.
 \end{lemma}

\begin{prop}\label{P:system}
If $(x_0,y_0,z_0,w_0)$ and $(x,y,z,w)$ are integer quadruples such that $\fr(x_0,y_0,z_0,w_0)=\fr(x,y,z,w)$,  there exists an integer $l $ such that the system
\begin{equation}\notag
\biggl\{
 \begin{array}{rcl}
 x_0p +(y_0+w_0)p'-lq' &=& xp + (y+w)p'  \\
z_0p+w_0 + lq &=& zp + w.
  \end{array}
  \end{equation}
is satisfied.
\end{prop}
\begin{proof} Using \eqref{E:definefracpq} and \eqref{E:d's}, we have
\begin{align}
\fr(x_0, y_0, z_0, w_0) &= x_0pq + y_0p'q + z_0pq' + w_0(p'q + q') \notag \\
&=q[x_0p + (y_0+w_0)p'] + q'[z_0p + w_0]. \label{E:square}
\end{align}
Since $q$ and $q'$ are relatively prime, we can apply Lemma~\ref{L:rigid} with $a$, $b$   being the two expressions in square brackets in \eqref{E:square}.
  \end{proof}

We now resume the proof of Proposition~\ref{P:2}.  Suppose that $\ga=\fr(x,y,z,w)$ with $x,y,z,w$ nonnegative, and further suppose (for a contradiction) that  $\kappa + \lambda \geq p$ (equivalently, $\kappa'-\kappa > 0$). Using Proposition~\ref{P:system} with $(x_0, y_0, z_0, w_0)=(p'-1, p-1, \kappa, -1)$, there exists an integer $l$ such that
\begin{equation}\label{E:gkappasystem}
\biggl\{
 \begin{array}{rcl}
 (p'-1)p +(p-2)p'-lq' &=& xp + (y+w)p'  \\
\kappa p-1 + lq &=& zp + w.
  \end{array}
  \end{equation}
    The second equation implies  $l \geq 0$, since  the right hand side is nonnegative (by assumption) and $\kappa p -1 < q$.    To simplify the system \eqref{E:gkappasystem}, we express $q'$ in terms of $p'$ and $p$.  We have
    \begin{align}
    q' &=\kappa'p'+\lambda' \notag \\
    & = (\kappa'-2\lambda')p' + 2\lambda'p'+\lambda' \notag \\
    &=(\kappa'-2\lambda')p'+\lambda'p, \notag
    \end{align}
where, at the last step, we use $p=2p'+1$. By Lemma~\ref{L:lambda'}, $2\lambda'=\kappa+\lambda -1 - (\kappa'-\kappa)(p-1)$, and hence
\begin{equation}\label{E:q'}
q'=Bp' + \lambda' p,
\end{equation}
where
\begin{equation}\label{E:B}
B = (\kappa'-\kappa)p - \lambda +1.
\end{equation}
Since $\kappa'-\kappa > 0$ and $\lambda < p$,  $B$ is positive.

  Using \eqref{E:q'},  the first equation of  \eqref{E:gkappasystem} becomes
\begin{equation}\label{E:eq1}
(p'-1-l\lambda')p + (p-2 - lB)p' = xp + (y+w)p'.
\end{equation}
Since $p$ and $p'$ are relatively prime,  Lemma~\ref{L:rigid} applies  to the left hand side, with $a$ and $b$ being the two expressions in parentheses.  Hence there exists $t \in {\Bbb Z}$ such that
\begin{align}
x&=p'-1-l\lambda' +tp' \notag \\
y+w &= p-2 - lB - tp.\label{E:<p-2}
\end{align}
The first equation implies $t\geq 0$ (otherwise $x<0$) and the second that $t\leq 0$ (otherwise $y+w < 0$).  Hence $t=0$.
Putting $q=\kappa p + \lambda$ into the second equation of  the system \eqref{E:gkappasystem}, we see that $w \equiv l\lambda - 1$ (mod $p$).  By \eqref{E:<p-2}, $y+w \leq p-2$ and in particular, since $y$ and $w$ are nonnegative, $w \leq p-2$.  The only possibility is $w=l\lambda - 1$. Hence
\begin{align}
y & = p-2-lB-l\lambda + 1 \notag \\
&= p-1-l[(\kappa'-\kappa)p + 1] \notag \\
& \geq p-1-l(p+1)
\end{align}
(since $\kappa'-\kappa >0$).  $y \geq 0$ requires $l=0$, and $w \geq 0$ requires $l>0$, a contradiction.

This completes the proof of Proposition~\ref{P:2}.  Combining this with Proposition~\ref{P:1}, we obtain  Theorem \ref{T1:Main}(ii).

\section{Representability of integers $>\gb$}\label{S:representability>g1}

In this section and the next  we recycle, as far as possible, the arguments in Sections~\ref{S:representability>g0} and \ref{S:repg0}, replacing $\ga$ by $\gb$.  Since $\gb = \ga - \lambda d_3$, we attempt this by simply reducing the fourth coordinate of each quadruple by $\lambda$.  The obstruction, of course, is that some of the fourth coordinates  thereby become negative.

\begin{prop}\label{P:5} If $\kappa + \lambda \leq p-\lambda$, all integers $>\gb$ are representable.
\end{prop}
\begin{proof}  It suffices to show that there is a nonnegative quadruple representing each integer in the closed interval  $[\gb + 1, \gb + d_1]$.  We start with the representation  $\gb + 1=\fr(0,p-1, \kappa-1, p-3-\lambda)$, obtained from \eqref{E:g0+1} by subtracting $\lambda$ from the fourth coordinate.  The fact that $\kappa+\lambda$ is odd and less than $p$, and that $\kappa \geq 1$, together imply that $\lambda \leq p-3$; thus this is a nonnegative representation.  Representing $t \in [0, d_1-1]$ by \eqref{E:paramdef},  we write $\gb+1+t$  $=\fr(x,y,z,w)$, where $x,y,z$ are given by \eqref{E:x}, \eqref{E:y}, \eqref{E:z}, respectively, and
\begin{equation}\label{E:w-lambda}
w=p-3- 2c+a+b-\lambda,
\end{equation}
which is obtained from \eqref{E:w} by subtracting $\lambda$.

If $z \geq 0$, the possible obstruction is $w<0$.  Then  $a \leq \lambda$ (since $p-3-2c \geq 0$).  Adding  \eqref{E:mainequiv} to $(x,y,z,w)$ yields a nonnegative quadruple, provided $c \geq \lambda'$.  If $c<\lambda'$,
\begin{align}
w &\geq p-3-2(\lambda' - 1) + a + b - \lambda \notag \\
&= p-3-(\kappa + \lambda - 3) + a + b - \lambda \notag \\
&\geq p-\lambda - (\kappa + \lambda) \notag \\
&\geq 0, \notag
\end{align}
contrary to the assumption that $w<0$.

 If $z<0$,  $b=\kappa$ and $c<\lambda'$.  We write $\gb + 1 + t = \fr(x',y',z',w')$, where $x',y',z'$ are given by \eqref{E:x'C}, \eqref{E:y'C}, \eqref{E:z'C}, respectively, and
\begin{equation}\label{E:w'C-lambda}
w'=2(\lambda' -c) - 2 + a -\lambda,
\end{equation}
which is obtained from \eqref{E:w'C} by subtracting $\lambda$.   The only possible obstruction is $w'<0$.  If this is the case, we add  \eqref{E:mainequiv} to $(x',y',z',w')$, yielding the quadruple
\begin{align}
x'' & = p'-2\lambda' + c \notag \\
y'' &= p-3-a \notag \\
z'' &= 2\kappa \notag \\
w'' &= 2(\lambda'-c) - 2 + a \notag
\end{align}
The assumption $w'<0$ implies $a  < \lambda \leq p-3$, so $y'' \geq 0$.  Clearly $z'' \geq 0$.  $w'' \geq 0$ since it is equal to $\eqref{E:w'C}$.  If $x''<0$ then $c<2\lambda' - p'$.  If this is the case,
\begin{align}
w' &\geq 2(\lambda' - (2\lambda' - p' -1))-2 + a -\lambda \notag \\
& \geq 2p'-2\lambda' - \lambda \notag \\
&=2p' - (\kappa + \lambda -1) - \lambda  \notag \\
&= p-\lambda - (\kappa + \lambda). \notag \\
&\geq 0, \notag
\end{align}
contradicting the assumption that $w'<0$.  Hence $(x'', y'', z'', w'')$ is a nonnegative quadruple.
\end{proof}

\begin{cor}\label{C:5}
 If $\kappa + \lambda \leq p- \lambda$ then $g \leq \gb$.
\end{cor}

To complete the proofs of Theorems~\ref{T4:Main} and \ref{T5:Main}, we need necessary and sufficient conditions for the representability of $\gb$, and conditions under which there is an integer $>\gb$ which is not representable.

\section{Representability of $\gb$ and $\gb +  \lambda'$}\label{S:repg1}
We need two preliminary results.
\begin{lemma}\label{L:lambda'}
$\lambda' = \displaystyle{\frac{\kappa + \lambda -1}{2} - (\kappa'-\kappa)p'}$.
\end{lemma}
\begin{proof}
By definition $q -\lambda = \kappa p$, from which we obtain
\begin{align}
(q-1) - \lambda &= \kappa (p-1) + \kappa -1 \notag \\
q' - \kappa p' &= \frac{\kappa + \lambda - 1}{2} \notag \\
q' - \kappa'p' &=  \frac{\kappa + \lambda - 1}{2}-(\kappa'-\kappa)p'. \notag
\end{align}
The left-hand side of the last equation is the definition of $\lambda'$.
\end{proof}

\begin{lemma}\label{L:prelims}
 If $\kappa + \lambda < p$, then: \quad (i) $\kappa = \kappa'$;\quad (ii) $\lambda' =  \displaystyle{\frac{\kappa + \lambda - 1}{2}} \geq 1$;\quad  (iii)
 $\displaystyle{\frac{p'}{\lambda'} < \frac{p}{\lambda}}.$
\end{lemma}
\begin{proof}
Using the formula for $\lambda'$  given in Lemma~\ref{L:lambda'} and the assumption that $\kappa + \lambda < p$, we have $\lambda' < p'-(\kappa' - \kappa)p'$.  Since $\kappa' - \kappa \geq 0$ and $\lambda' \geq 0$, the only possibility is  (i).  The  equality in (ii) follows from (i) and Lemma~\ref{L:lambda'}.  The right-hand inequality follows from  $\kappa + \lambda \geq 3$.   To prove (iii), suppose that $\frac{p'}{\lambda'} \geq \frac{p}{\lambda}$.  Then $p'\lambda \geq p \lambda'$, and hence, using (ii),
\begin{align}
2p'\lambda &\geq p(2\lambda ') \notag \\
(p-1)\lambda &\geq p(\kappa + \lambda - 1) \notag \\
-\lambda &\geq p(\kappa - 1) \geq 0, \notag
\end{align}
a contradiction.
\end{proof}

 Suppose that $\gb=\fr(x,y,z,w)$ with $x,y,z,w$ nonnegative.   Using Proposition~\ref{P:system} with $(x_0, y_0, z_0, w_0)=(p'-1, p-1, \kappa, -1-\lambda)$, there exists an integer $l$ such that
\begin{equation}\label{E:2ndsystem}
\biggl\{
 \begin{array}{rcl}
 (p'-1)p +(p-2)p'-lq' &=& xp + (y+w)p'  \\
\kappa p-1- \lambda  + lq &=& zp + w.
  \end{array}
  \end{equation}
The second equation implies $l \geq 0$ (and the first  that $l$ cannot be too large).  Imitating the argument leading from \eqref{E:gkappasystem} to \eqref{E:<p-2},  we see that there exists an integer $t \geq 0$ such that
\begin{align}
x&=p'-1-l\lambda' +tp' \label{E:(a)} \\
y+w &= p-2 - l +\lambda (l-1) - tp. \label{E:(bb)}
\end{align}
(We used $B$ as defined at \eqref{E:B}, but with $\kappa'-\kappa = 0$.)  From the second equation of   \eqref{E:2ndsystem} (putting $q=\kappa p + \lambda$),   we see that $w \equiv (l-1)\lambda - 1$ (mod $p$).  Then \eqref{E:(bb)} yields $y \equiv p-1-l$ (mod $p$).  Hence there exist $\mu, \nu \in {\Bbb Z}$ such that
\begin{align}
y &= \nu p + p-1 - l \notag \\
w & = \mu p + (l-1)\lambda - 1.  \label{E:wdef}
\end{align}
By \eqref{E:(bb)}, $\mu + \nu = -t$.   $\nu \geq 0$ from the assumption that $y\geq 0$.  Provided that $l \leq p-1$ (we shall see shortly that this assumption is justified), we may add a suitable multiple of \eqref{E:trivequiv} to $(x,y,z,w)$, and so assume $\nu = 0$.  Then $\mu = -t$.  From \eqref{E:wdef} and the second equation of \eqref{E:2ndsystem},
\begin{align}
z & = (l+1)\kappa +t. \notag
\end{align}

Thus  a quadruple representing $\gb$ has the general form
\begin{equation}\label{E:g1rep}
\begin{array}{rcl}
x  &=& p'-1-l\lambda' +tp' \qquad (0 \leq t, \quad 0\leq l \leq p-1) \\
y  &=& p-(l+1)  \\
z  &=&(l+1)\kappa + t \\
w & =& -tp + (l-1)\lambda - 1.
\end{array}
\end{equation}
($t=l=0$ yields the defining representation  of $\gb$.)

\begin{prop}\label{P:3} $\gb$ is representable if and only if the pair is of Type I.
\end{prop}

\begin{proof}
If the pair is of Type I, let $t=\tau$ and $l=\tau + 2$ in \eqref{E:g1rep}.
Then $x =(\tau +1)p' - (\tau + 2) \lambda'-1$ is nonnegative as a  consequence of $\frac{\tau+2}{\tau+1} <\frac{p'}{\lambda'}$, and $w=-\tau p + (\tau + 1) \lambda -1$ is nonnegative as a consequence of  $\frac{p}{\lambda} < \frac{\tau+1}{\tau}$.  Obviously $z \geq 0$.  It remains only to verify that $l \leq p-1$, so that $y \geq 0$.   $\kappa + \lambda < p$ implies  $\lambda \leq p-3$, and $\tau < \lambda$, so $l=\tau+2<\lambda + 2 \leq p-1$.

Suppose the pair is of Type II and $t$, $l$ are nonnegative integers making \eqref{E:g1rep} a nonnegative quadruple.    $x \geq 0$, $w \geq 0$ imply, respectively,
\begin{equation}\notag
\frac{l}{t+1}<\frac{p'}{\lambda'}, \quad\text{and}\quad
\frac{p}{\lambda} < \frac{l-1}{t}.
\end{equation}
It follows that $l > t+1$, and in particular,
\begin{equation}\label{E:squeeze}
\frac{t+2}{t+1} \leq \frac{l}{t+1}<\frac{p'}{\lambda'}, \quad\text{and}\quad
\frac{p}{\lambda} < \frac{t+1}{t}<\frac{l-1}{t}.
\end{equation}
Since the pair is of Type II, the left-hand inequality  implies $t > \tau$, while the right-hand inequality, by the definition of $\tau$, implies  that $t \leq \tau$, a contradiction.
\end{proof}

From Corollary~\ref{C:5} and Proposition~\ref{P:3}, we obtain

\begin{cor}\label{C:6}
 If the pair is of Type II  with $\kappa + \lambda \leq p-\lambda$, then $g=\gb$.  If the pair is of Type I  with $\kappa + \lambda \leq p-\lambda$, then $g<\gb$.
 \end{cor}

The next proposition treats the remaining Type II pairs, and completes the proofs of all statements regarding $\ga$ and $\gb$ in Theorems~\ref{T1:Main}, \ref{T4:Main} and \ref{T5:Main}.

\begin{prop}\label{P:4} If the pair is of Type II with $\kappa + \lambda>p - \lambda$, then $\gb + \lambda'$ is not representable and hence $g>\gb$.
\end{prop}
\begin{proof}
 Suppose  $\gb + \lambda' = \fr(x,y,z,w)$ for a nonnegative quadruple $(x,y,z,w)$.  A general form for $(x,y,z,w)$ is produced from  \eqref{E:g1rep} by using \eqref{E:fact1} to write
$
\gb + \lambda' = \gb + \lambda'\cdot  \fr(1,0,0,-2).
$
Reducing the fourth coordinate of \eqref{E:g1rep} by $2\lambda'=\kappa + \lambda -1$ (Lemma~\ref{L:prelims}(ii)), and increasing the first by $\lambda'$, we obtain
\begin{equation}\label{E:g1+lambdaprimerep}
\begin{array}{rcl}
x  &=& p'-1-(l-1)\lambda' +tp' \qquad (0 \leq t, \quad 0\leq l \leq p-1) \\
y  &=& p-(l+1)  \\
z  &=&(l+1)\kappa + t \\
w & =& -tp + (l-2)\lambda - \kappa.
\end{array}
\end{equation}
 The assumptions $x \geq 0$ and $w \geq 0$ imply almost the same inequalities as at \eqref{E:squeeze}, except that $l$ is  replaced $l-1$ where it occurs.  Regardless, we arrive at the same contradiction ($t>\tau$ and $t \leq \tau$) which concluded the proof of Proposition~\ref{P:3}.
 \end{proof}

\begin{rem}\label{R:white}
For the remaining Type I pairs (having $\kappa+\lambda > p-\lambda$ and colored white in Figure 1), both $g<\gb$ and $g>\gb$ are possible.  A patient reader can verify, for example,  that $g< \gb$ for the pair $(11, 17)$ and $g > \gb$ for the pair $(29, 103)$.
\end{rem}

\section{The lower bound}

It remains to prove that $\gc$ is a universal lower bound on the Frobenius number, and that it is sharp if $p=3$ or if $(p,q)$ is a twin prime pair.
\begin{prop} $\gc$ is not representable for any pair with $\kappa + \lambda < p$.
\end{prop}
\begin{proof}
Suppose $(p,q)$ is a pair for which $\gc$ is representable.   Using Proposition~\ref{P:system} with $(x_0, y_0, z_0, w_0)=(p'-1, p-1, \kappa, 2-p)$, there exists an integer $l$ such that
\begin{equation}\label{E:lowerboundsystem}
\biggl\{
 \begin{array}{rcl}
 (p'-1)p +p'-lq' &=& xp + (y+w)p'  \\
\kappa p +2-p+ lq &=& zp + w,
  \end{array}
  \end{equation}
    for nonnegative integers $x,y,z,w$.  The second equation implies  $l \geq 0$.    Collecting the multiples of $p$ and the multiples of $p'$ on the left-hand side of the first equation, and using \eqref{E:q'} and \eqref{E:B} with $\kappa = \kappa'$, and Lemma~\ref{L:rigid}, we see that there exists  $t \in {\Bbb Z}$ such that
\begin{align}
x&=p'-1-l\lambda' +tp' \label{E:xand}  \\
y+w &= 1 + l(\lambda-1)-tp.\label{E:y+w}
\end{align}
\eqref{E:xand} implies $t\geq 0$.
Putting $q=\kappa p + \lambda$ into the second equation of \eqref{E:lowerboundsystem},
\begin{equation}\label{E:zp+w}
p(\kappa(l+1)-1) + l\lambda + 2 = zp + w.
\end{equation}
It follows that  $w \equiv l\lambda +2$ (mod $p$), and, using  \eqref{E:y+w}, that $y \equiv -(l+1)$ (mod $p$).  Hence there exist $\mu, \nu \in {\Bbb Z}$ such that
\begin{align}\notag
y & = \nu p - (l+1) \qquad (\nu >0) \notag \\
w &= \mu p + l\lambda + 2. \notag \\
\end{align}
By  \eqref{E:y+w}, $\mu = -t + \nu$.  From \eqref{E:zp+w}, $z=\kappa(l+1) -1 + t + \nu.$

    Thus  a quadruple representing $\gc$ has the general form
\begin{equation}\label{E:g2rep}
\begin{array}{rcl}
x  &=& p'-1-l\lambda' +tp' \qquad (t, l \geq 0) \\
y  &=& \nu p-(l+1) \qquad \qquad (\nu > 0) \\
z  &=&\kappa(l+1)-1 + t+\nu \\
w & =& -(t+ \nu)p + l\lambda +2.
\end{array}
\end{equation}
($t=l=0$, $\nu = 1$ yields the defining representation of $\gc$.)  The requirements $x \geq 0$ and $w \geq 0$ imply
\begin{equation}\notag
\frac{(t+\nu)p - 2}{\lambda} \leq l \leq \frac{(t+1)p' - 1}{\lambda'}.
\end{equation}
We show that this leads to a contradiction.  Minimizing the left hand member of the inequality by taking $\nu = 1$, we obtain
\begin{equation}\label{E:g2ineq}
\frac{(t+1)p - 2}{\lambda} \leq \frac{(t+1)p' - 1}{\lambda'}.
\end{equation}
Since $\kappa + \lambda < p$, $\lambda' = (\lambda + (\kappa - 1))/2 \geq \lambda/2$,  and hence $1/\lambda' \leq 2/\lambda$, with equality if and only if $\kappa = 1$.  Rearranging  \eqref{E:g2ineq}, we obtain
$$
(t+1)\biggr(\frac{p}{\lambda} - \frac{p'}{\lambda'}\biggr) \leq \biggl(\frac{2}{\lambda} - \frac{1}{\lambda'}\biggr),
$$
 which is a contradiction if $\kappa =1$, since then the left-hand side is positive (Lemma~\ref{L:prelims} (iii)), while  the right-hand side is $0$.  Hence assume $\kappa>1$,   and multiply both sides by $\lambda\lambda'>0$.  This yields
$$
 (t+1)(p\lambda' - p'\lambda) \leq 2\lambda' - \lambda.
$$
    The right hand side is equal to $\kappa - 1>0$, and the left-hand side can be rewritten as
    $$
    \frac{1}{2}(t+1)(p(\kappa + \lambda - 1)-(p-1)\lambda),
    $$
    which simplifies to
   $$
     \frac{1}{2}(t+1)(p(\kappa  - 1)+\lambda).
   $$
    Thus we have
   $$
    \frac{1}{2}(t+1)(p(\kappa  - 1)) < \kappa - 1.
    $$
    Canceling the non-zero factor $\kappa -1$ leads to the contradiction
    $$(t+1)p < 2.$$
\end{proof}

Thus $\gc \leq g$ for all pairs.  The bound is attained if $p=3$, by Theorem~\ref{T1:Main} (iii).  The next proposition shows that the bound is also attained for twin prime pairs.

\begin{prop}\label{P:twinrep} If $(p,q)$ is a twin prime pair, all integers $>\gc$ are representable.
\end{prop}
\begin{proof}  We adapt the proof of Proposition~\ref{P:1} (cf. Proposition~\ref{P:5}).  It suffices to show that the integers in the closed interval $[\gc+1, \gc +d_1]$ are representable.  We start with the representation  $\gc + 1=\fr(0,p-1, 0, 0)$, obtained from \eqref{E:g0+1} by subtracting $p-3$ from the fourth coordinate, and using the fact that $\kappa =1$.   For twin pairs, $q'=p'+1$ and hence, from \eqref{E:factq'} and \eqref{E:factp'}, we derive
\begin{equation}
\fr(0, -1,1,0)  = 1. \label{E:facttwin}
\end{equation}
Let an integer $t \in [0, d_1-1]$ be represented with the division algorithm as
\begin{equation}\notag
t= aq' + b + c,\quad \text{with} \quad 0 \leq a \leq p-1 \text{\ ($a$ \ maximal),}\quad b,c \geq 0, \quad b+c \leq p'.
\end{equation}
  The bound on $b+c$  comes from the maximality of $a$ and the fact that $q' = p' + 1$.  $a$ and $b+c$ are uniquely determined by  $t$ and conversely.   It follows from  \eqref{E:factq'}, \eqref{E:fact1} and \eqref{E:facttwin} that, for $t \in [0, d_1-1]$,
 $$
\gc + 1 + t = \fr(0, p-1, 0,0) + a\cdot\fr(0,-1,0,1) + b \cdot \fr(0,-1,1,0) + c\cdot \fr(1,0,0,-2).
$$
Thus $\gc+1+t =\fr(x,\ y,\ z,\ w)$,   where
\begin{align}
x  &= c\notag\\
y  &= p-1-(a+b)\notag\\
z  &= b \notag \\
w &  = a-2c. \notag
\end{align}
If $a\leq p'$, then $p-1-a \geq p'$ and we may assume $c=0$ by increasing $b$, if necessary, while maintaining $y \geq 0$.  In fact,  $y \geq p-1-(p' + p') = 0$, $w=a\geq 0$, and we have a  nonnegative quadruple.  Hence suppose   $a = p' + s$, $1 \leq s \leq p'$.  Let $b=p'-i$ and $c=p'-k$, $0 \leq i,k \leq p'$.    Since $b+c \leq p'$, $i+k \geq p'$.  We claim there is a choice of $i$ and $k$ making $(x,y,z,w)$ a non-negative quadruple.  Clearly $x,z \geq 0$ for all choices of $i,k$.  If $i \geq s$, $w=s+2k - p'$  $\geq i+k + k-p'$ $\geq p'+k-p' = k \geq 0$.  If $i < s$, put $i'=i + (s-i) = s$ and $k'=k-(s-i)$, so that  $i'+k' = i+k \geq p'$.  Let $b=p'-i'$ and $c=p'-k'$.  Then
\begin{align}
x  &= p'-k'>p'-k \geq 0\notag\\
y  &= i'-s = 0\notag\\
z  &= p'-i' = p'-s \geq 0 \notag \\
w &  = i'+2k'-p' \notag \\
& =i' + k' + k'-p' \notag \\
& \geq p'+k'-p' =k' \geq 0.\notag
&
\end{align}
\end{proof}

\begin{rem} We conjecture that $g=\gc$  {\rm  only if} $(p,q)$ is a twin prime pair.
\end{rem}

This completes the proofs of Theorems~\ref{T1:Main}, \ref{T4:Main} and \ref{T5:Main}.

\section{The largest non-genus of ${\mathbb Z}_{pq}$}\label{S:surfaces}

We return to the motivating question of determining the largest non genus $\nu_{pq}$ of  a ${\mathbb Z}_{pq}$ action (Problem II, Section~\ref{ss:groups}, $n=pq$).   We show that given $p>3$, the lower bound in \eqref{gpqbound} (and \eqref{E:groupineq}) is attained for all but finitely many $q>p$. This is Theorem \ref{T6:main} which we restate here.
\begin{theorem}\label{T:cyclic}
For primes $3<p<q$, with $\kappa + \lambda \geq p$ and $q \neq 2p-1,  3p-2$, the largest non-genus of ${\mathbb Z}_{pq}$ is
$$\nu_{pq} = \ga - pq + 1,$$
where $\ga$ is the integer defined at \eqref{E:keynums},  equal to the Frobenius number $g(\{d_0, d_1, d_2, d_3\})$.
\end{theorem}
\begin{proof} We re-visit the argument in Section~\ref{S:representability>g0}, showing that the quadruples constructed there satisfy the admissibility conditions required by Lemma~\ref{L:admissibility}, or can be  altered  (by adding  integer linear combinations of the vectors $e_0, e_1, e_2$),  so as to satisfy them.  Bringing the notation in Lemma~\ref{L:admissibility} into accord with that introduced in Sections~\ref{S:intro} and \ref{S:results},   we put $p=p_1$, $q=p_2$,  and use $x,y,z,w$ and $d_0, d_1, d_2, d_3$ instead of $x_{00} ,x_{10}, x_{01},  x_{11}$, and  $d_{00}, d_{10}, d_{01}, d_{11}$,  respectively.     In this notation,  conditions \eqref{E:bincond1} and \eqref{E:bincond2} are
 \begin{align}
  y+w &\ne 1, \qquad  z+w \ne 1, \label{cond1}\\
     x+y+w & \ne 0, \qquad  x+z+w \ne 0. \label{cond2}
 \end{align}
It is convenient to replace the condition $y+w \ne 1$ with the stronger condition $y+w>1$.  A nonnegative quadruple satisfying \eqref{cond1}, \eqref{cond2} and $y+w \neq 0$  will be called  {\it strongly admissible}.  The extra condition is
 imposed so that if $(x,y,z,w)$ is strongly admissible, then $(x, y+1, z,w)$ is admissible.  With this guarantee, it is sufficient to produce strongly admissible representations of the integers in the closed interval $[\ga + 1, \ga + d_1]$.

 Suppose first that the quadruple $(x,y,z,w)$  as defined by \eqref{E:x} - \eqref{E:w} is nonnegative, that is,  assume $z\geq 0$.  One easily verifies that $x+y+w \geq p-1> 0$, $x+z+w \geq \kappa >0$, $y+w \geq p-1>1$.  It remains  to consider the possibility that
 \begin{equation}\notag
 z+w=\kappa - 1-b + p-3-2c+a+b = 1.
 \end{equation}
 This occurs if and only if
 \begin{enumerate}
 \item $\kappa = 1$ and $(a,b,c)=(1,b,p'-1)$; or
\item $\kappa = 2$ and $(a,b,c)=(0,b, p'-1)$.
  \end{enumerate}
 $b \leq \kappa-1$ by \eqref{E:z} and the assumption $z\geq 0$.
  Thus in (i), $b=0$.  The triple $(1,0,p'-1)$ corresponds to the inadmissible quadruple $(p'-1, p-2, 0, 1)$.  $\kappa = 1$ is equivalent to $p+2 \leq q \leq 2p-1$ or $q' \leq 2p'$.  We have excluded $q=2p-1$, so we may assume $q'<2p'$.
  It is easily verified that  $(p'-1, p-2, 0, 1) +
  e(0,0)$  is strongly admissible.  In case (ii),  $b=0$ or $1$ and  the two triples $(0,0,p'-1)$ and $(0,1,p'-1)$ correspond to the inadmissible quadruples
 \begin{equation}\label{E:2quads}
 (p'-1, p-1, 1, 0) \quad\text{and} \quad (p'-1, p-1,  0, 1),
  \end{equation}
  respectively.
 $\kappa =2$ is equivalent to $2p+1 \leq q \leq 3p-2$ or $p \leq q'\leq 3p'$.  Since we have excluded $q=3p-2$, we may assume $q'<3p'$.   Addition of $e(0,1)$ makes both quadruples in \eqref{E:2quads} strongly admissible.

 Now assume that $z<0$ in the quadruple $(x,y,z,w)$ defined at \eqref{E:x}-\eqref{E:w}.  We re-visit the proofs of Lemmas~\ref{L:A}, \ref{L:B} and \ref{L:C}.

 If $s <\kappa'-\kappa-1$, Lemma~\ref{L:A}(ii) produces the nonnegative quadruple
 \begin{align}
 x' &=c+(s+1)p' \notag\\
 y' &=p-1-a  \notag\\
 z' &=0  \notag \\
 w' &= p-3-2c +a +\kappa + s -(s+1)p. \notag
 \end{align}
 If this is inadmissible, $w'\leq 1$.  By \eqref{E:p-3-2c} and Lemma~\ref{L:A}(i),  $w'$ is the sum of  three nonnegative quantities:  $(p-3-2c)$, $a$, and $\kappa+s-(s+1)p$.    Since $p-3-2c$ is even, $w'\leq 1$ implies $p-3-2c=0$, equivalently, $c=p'-1$.  The other two quantities are either both $0$, or one is $0$ and the other $1$.  This yields three possible quadruples:  $(x', p-1, 0,0)$, which is strongly admissible, and
 \begin{equation}\label{E:three}
 (x', p-1-a, 0, 1), \quad a=0, 1,
 \end{equation}
 where $x'=p'-1 + (s+1)p' .$   We show that these two quadruples cannot arise under the assumed conditions.  They  are supposed to represent the integers
 \begin{equation}\notag
 \ga + aq'+ (\kappa + s)p' + p'-1, \quad a = 0, 1.
 \end{equation}
  Equating these two integers with the corresponding values of $\fr$ on the two quadruples in \eqref{E:three},  we obtain, for $a=0,1$,
 \begin{equation}\notag
 \fr(p'-1, p-1, \kappa, -1) + aq' + (\kappa + s)p' + p'-1 = \fr(p'-1 + (s+1)p', p-1-a, 0, 1).
 \end{equation}
By the linearity of $\fr$ this is equivalent to
\begin{align}
0 &=\fr(-(s+1)p', a, \kappa, -2) + aq' +(\kappa + s)p' + p'-1 \notag \\
&=-(s+1)p'd_0 +a(d_1 + q') + \kappa(d_2 + p') -2d_3 + (s+1)p'-1\notag \\
&=-(s+1)p'(d_0 - 1) + a(d_1+q') + \kappa(d_2 + p') - 2d_3 - 1. \label{E:last}
\end{align}
 The identities
\begin{equation}\label{E:identities}
d_1+q'=d_2+p' = d_3 \quad\text{and}\quad 2d_3 +1= d_0
\end{equation}
follow easily from the definitions of the $d_i$'s at \eqref{E:d's}.  Thus \eqref{E:last} is equivalent to
\begin{align}
d_0 &=-(s+1)p'(2d_3) + (a+\kappa)d_3 \notag\\
 &= (a + \kappa -2p'(s+1))d_3 \notag\\
 &=(a + \kappa - (p-1)(s+1))d_3\notag \\
 &=(a + \kappa +s - (s+1)p + 1)d_3.\label{E:last2}
\end{align}
The expression in parentheses on the right is equal to $w'+1=2$.  Thus \eqref{E:last2} is equivalent to
$d_0 = 2d_3$, contradicting the last identity in \eqref{E:identities}.

If $s=\kappa'-\kappa -1$, Lemma~\ref{L:B}(ii) produces the nonnegative quadruple
\begin{align}
 x' &=c-\lambda' \notag \\
 y' &=p-2-a \notag \\
 z' &=\kappa \notag\\
 w' &= p-3-2c +a +[\kappa'-1-(\kappa'-\kappa)p + \lambda]  \label{E:w'B}
 \end{align}
 We claim that the quantity in brackets in \eqref{E:w'B} is nonnegative.  This is a consequence of
 \begin{align}
 q' &\geq \kappa'p' \notag \\
 q-1 &\geq \kappa'(p-1) \notag \\
 \kappa'-1 &\geq \kappa'p - q \notag \\
 &=(\kappa'-\kappa)p + \kappa p - q \notag \\
 &=(\kappa'-\kappa)p - \lambda. \label{E:laststep}
 \end{align}
It follows that \eqref{E:w'B}  is the sum of the three nonnegative quantities
 $(p-3-2c)$, $a$, and $\kappa'-1 - (\kappa'-\kappa)p + \lambda$.  If the quadruple is inadmissible, $w' \leq 1$ and hence the even number $p-3-2c=0$, or equivalently, $c=p'-1$. The other two quantities are either both $0$, or one of them is $1$ and the other is $0$.   If $\kappa'-1 - (\kappa'-\kappa)p + \lambda = 1$, then $a=0$, and we have the strongly admissible quadruple $
(p'-1-\lambda', p-2, \kappa, 1)$.
  If $\kappa'-1 - (\kappa'-\kappa)p + \lambda = 0$, then reversing the chain of inequalities ending at \eqref{E:laststep}, with inequalities replaced by equalities,  $q'=\kappa'p'$.  In that case, $\lambda'=0$ and we have the quadruples
  \begin{equation}\label{E:two}
(p'-1, p-2-a, \kappa, 0),  \quad a=0,1,
\end{equation}
which are strongly admissible if $\kappa > 1$  ($p>3$ is required here).
If $\kappa = 1$, the quadruples in \eqref{E:two}  are inadmissible, but   $\kappa + \lambda \geq p$ implies $\lambda = p-1$ and $q=2p-1$, which is excluded.

If $s=\kappa'-\kappa$,  Lemma~\ref{L:C} produces the nonnegative quadruple \eqref{E:x'C} -\eqref{E:w'C}.  If this is inadmissible, $w' \leq 1$.  By \eqref{E:lambdaC}, $w'$ is the sum of two nonnegative quantities, $2(\lambda'-c) -2$ and $a$.  Since the former is even, it must be $0$.  Equivalently, $c=\lambda'-1$.  Hence there are two quadruples corresponding to $a=0,1$:
\begin{equation}\label{E:twomore}
(p'-1, p-2, \kappa, 0),  \quad (p'-1, p-3, \kappa, 1).
\end{equation}
The latter is strongly admissible ($p>3$ is required), as is the former if $\kappa > 1$.  If $\kappa =1$, $(p'-1, p-2, \kappa, 0)$  is inadmissible.  Then $q' \leq 2p'$, and in fact, $q'<2p'$ since $q \ne 2p-1$.  In that case, addition of $e(0,0)$ yields a strongly admissible quadruple.
This concludes the proof of Theorem~\ref{T:cyclic}.
\end{proof}

For a twin prime pair, it is not difficult to show that the integer $\fr(0,p-1,1,0)$ has no other nonnegative representation, and hence no admissible representation.  A straightforward argument, similar to the proof of Proposition~\ref{P:twinrep}, shows that the next $d_1$ integers  all have strongly admissible representations.    This yields the following theorem, whose proof is omitted.

\begin{theorem}\label{T:twingroup} For a twin prime pair $(p,q)$, $p>3$, the largest non-genus of ${\mathbb Z}_{pq}$ is
$$\nu_{pq} = \fr(0, p-1, 1, 0) - pq+1 = \gc +1+d_2 - pq+1.$$
\end{theorem}

\noindent  Hence for twin prime pairs, $\nu_{pq}$ is about midway between the bounds of \eqref{gpqbound}.

\end{document}